\documentclass[10pt]{article}

\usepackage{fullpage}
\usepackage[utf8]{inputenc}		

\usepackage{enumerate}

\usepackage{amsmath}
\usepackage{amsthm}
\usepackage{amsfonts}
\usepackage{amssymb}
\usepackage{mathtools}

\usepackage{pgf}

\usepackage{subfigure}

\usepackage[pdftex,pdfborder={0 0 0},
colorlinks=true,
linkcolor=blue,
citecolor=red,
pagebackref=true,
]{hyperref}

\newtheorem{theorem}{Theorem}[section]
\newtheorem{definition}{Definition}[section]
\newtheorem{lemma}{Lemma}[section]
\newtheorem{proposition}{Proposition}[section]
\newtheorem{remark}{Remark}[section]
\newtheorem{corollary}{Corollary}[section]

\def \to{\rightarrow}

\def \norm{\|}
\def \abs{|}
\def \X{\mathcal{X}}

\def \E {\mathbb{E}}
\def \R {\mathbb{R}}
\def \N {\mathbb{N}}

\def \C {\mathcal{C}}

\def \B {\mathcal{B}}

\def \mes{\mathcal{P}}

\def \mystates{\mathcal{S}}

\def \d{\mathrm{d}}
\def \argmin{\mathrm{argmin}}

\def \e{\eta}
\def \y{\gamma}

\def \pos{x}
\def \posn{y}

\def \dt{\Delta t}
\def \time{\mathcal{T}}

\def \transition{\mathcal{K}_{\mystates,\time}}
\def \powH{q'}
\def \powL{q}
\def \convexset{\mathcal{C}}
\def \convexmset{\mes(\convexset)}

\def \varinC{x}

\title{Finite mean field games: fictitious play and convergence to a first order continuous mean field game}
\author{Saeed Hadikhanloo\thanks{CMAP, \'Ecole Polytechnique, CNRS, Universit\'e Paris Saclay, and INRIA, France (saeed.hadikhanloo@inria.fr).}  \and 
 Francisco J. Silva\thanks{TSE-R, UMR-CNRS 5314,  Universit\'e Toulouse I Capitole, 31015 Toulouse, France,  and Institut de recherche XLIM-DMI, UMR-CNRS 7252, Facult\'e des sciences et techniques, Universit\'e de Limoges, 87060 Limoges, France (francisco.silva@unilim.fr).}}
\date{}

\begin{document}
\maketitle

\begin{abstract}
In this article we consider finite Mean Field Games (MFGs), i.e. with finite time and finite states. We adopt the framework introduced in \cite{GomeMohrSouza} and study two seemly unexplored subjects. In the first one, we analyze the convergence of the fictitious play learning procedure, inspired by the results in continuous MFGs (see \cite{CDHD2015} and \cite{hadikhanloo2017learning}). In the second one, we consider the relation of some finite MFGs and continuous first order MFGs. Namely, given a continuous first order MFG problem and a sequence of refined space/time grids, we construct a sequence finite MFGs whose solutions admit limits points and every such limit point solves the continuous first order MFG problem.
\end{abstract}
\section{Introduction} 
Mean Field Games (MFGs)  were introduced by Lasry and Lions  in  \cite{LL06cr1,LL06cr2,LL07mf} and, independently, by Huang, Caines and Malham\'e in \cite{HCMieeeAC06}. One of the main purposes of the theory is to develop a notion of Nash equilibria for dynamic games, which can be deterministic or stochastic, with an infinite number of players. More precisely, if we consider a $N$-player game and we assume that the players are indistinguishable and small, in the sense that a change of strategy of player $j$ has a small impact on the cost for player $i$, then, under some assumptions, it is possible to show that as $N\to \infty$ the sequence of equilibria  admits limit points (see \cite{CDLL}). The latter correspond to  probability measures on the set of actions and define the notion of equilibria with a continuum of agents. An interesting feature of the theory is that it allows to obtain important qualitative information on the equilibria and the resulting problem is amenable to numerical computation. We refer the reader to the lessons by P.-L. Lions \cite{cursolions} and to \cite{CDnotesonMFG,GLL2010,MR3195844,MR3559742} for surveys on the theory and its applications.

Most of the literature about MFGs deals with games in continuous time and where the agents are distributed on a continuum of states (see \cite{CDnotesonMFG}). In this article we consider a MFG problem where the number of states and times are finite. For the sake of simplicity, we will call {\it finite MFGs} the games of this type. This framework has been introduced by Gomes, Mohr and Souza in \cite{GomeMohrSouza}, where the authors prove results related to the existence and uniqueness of equilibria, as well as the convergence to a stationary equilibrium as time goes to infinity.  

Our contribution to these type of games is twofold. First, we analyze  the {\it fictitious play} procedure, which is a learning method for computing Nash equilibria in classical game theory, introduced by Brown in \cite{BrownFictitiousPlay}. We refer the reader to \cite[Chapter 2]{FL98} and the references therein for a survey on this subject. Loosely speaking,  the idea is that at each iteration, a typical player  implements a best response strategy to his {\it belief} on the action of the remaining players. The belief at iteration $n\in \N$ is given, by definition, by the average of outputs of decisions of the remaining players in the previous iterations $1, \hdots, n-1$. In the context of continuous MFGs,  the study of the convergence of such procedure to an equilibrium has been first addressed in \cite{CDHD2015}, for a particular class of MFGs called {\it potential MFGs}. This analysis has then been extended in \cite{hadikhanloo2017learning}, by assuming that the MFG is monotone, which means that agents have aversion to imitate the strategies of other players. Under an analogous monotonicity assumption, we prove in Theorem \ref{FirstorderMFGaslimitsoffiniteMFG} that the fictitious play procedure converges also in the case of finite MFGs.

Our second contribution concerns the relation between continuous and finite MFGs. We consider here a first order continuous MFG and we associate to it a family of finite MFGs defined on finite space/time grids. By applying the results in \cite{GomeMohrSouza}, we know that for any fixed space/time grid the associated finite MFG admits at least one solution. Moreover, any such solution induces a probability measure on the space of strategies. Letting the grid length tend  to zero, we prove that the aforementioned sequence of probability measures is precompact and, hence, has at least one limit point. The main result of this article is given in Theorem \ref{ConvergenceMainTheorem} and asserts that any such limit point is  an equilibrium of the continuous MFG problem. To the best of our knowledge, this is the first result relating the equilibria for continuous MFGs, introduced in \cite{LL07mf}, with the equilibria for finite MFGs, introduced in \cite{GomeMohrSouza}. 

The article is organized as follows. In Section \ref{preliminaries} we recall the finite MFG introduced in \cite{GomeMohrSouza} and we state our first assumption that ensures the existence of at least one equilibrium. In Section \ref{Study_of_the_discrete_system} we describe the fictitious play procedure for the finite MFG and prove its convergence under a monotonicity assumption on the data. In Section \ref{FirstorderMFGaslimitsoffiniteMFG} we introduce the first order continuous MFG under study, as well as the corresponding space/time discretization and the associated finite MFGs. As the length of the space/time grid tends to zero, we prove several asymptotic properties of the  finite MFGs equilibria   and we also prove our main result showing their convergence to a solution of the continuous MFG problem. \medskip\\

{\bf Acknowledgements}: The second author acknowledges financial support by the ANR project MFG ANR-16-CE40-0015-01 and the PEPS-INSMI Jeunes project “Some
open problems in Mean Field Games” for the years 2016 and 2017. Both authors acknowledge financial support by the PGMO project VarPDEMFG.
\section{The finite state and discrete time Mean Field Game problem}\label{preliminaries}

We begin this section by presenting the MFG problem introduced in \cite{GomeMohrSouza} with finite state and discrete time. Let  $\mystates$ be a finite set, and let $\time = \{ 0, \hdots , N\}$.
We denote by $|\mystates|$ the number of elements in $\mystates$, and by 
$$\mes(\mystates):= \left\{ m: \mystates \to [0,1] \; \big| \; \sum_{x \in \mystates} m(x) = 1 \right\},$$
the simplex in $\R^{|\mystates|}$, which is identified with the set of probability measures over $\mystates$. We define now the notion of transition kernel associated to $\mystates$ and $\time$. 
\begin{definition}\label{definition_transition_kernel}
We denote by $\mathcal{K}_{\mystates,\time}$ the set of all maps $P : \mystates \times \mystates \times \left(\time \setminus \{ N \}\right) \to [0,1]$, called the transition kernels, such that $P(x,\cdot,k)\in \mes(\mystates)$ for all $x\in \mystates$ and $k\in \time \setminus \{ N \}$.
\end{definition}
Note that $\mathcal{K}_{\mystates,\time}$ can be seen as a compact subset of $\R^{|\mystates| \times |\mystates| \times N}$.  Given an initial distribution $M_0 \in \mes(\mystates)$ and $P \in \mathcal{K}_{\mystates,\time}$, the pair $(M_0, P)$ induces a probability distribution over $\mystates^{N+1}$, with marginal distributions given by 
\begin{equation}\label{definition_M_P}
\begin{array}{rcl}
M_{P}^{M_0}(x_{0},0)&:=& M_0(x_{0}), \hspace{0.3cm} \forall \;  x_0 \in \mystates, \\[6pt]
M_{P}^{M_0}(x_{k}, k)&:=& \sum_{(x_0, x_1, \hdots, x_{k-1}) \in \mystates^{k}} M_0(x_0)\prod_{k'=0}^{k-1} P(x_{k'},x_{k'+1},t_{k'})  \hspace{0.3cm} \forall \; k=1,\hdots,N, \; \; x_k \in \mystates,
\end{array}
\end{equation}
or equivalently, written in a recursively form, 
\begin{equation}\label{definition_M_P_recursive}
\begin{array}{rcl}
M_{P}^{M_0}(x_{k},0)&:=& M_0(x_{0}), \hspace{0.3cm} \forall \;  x_0 \in \mystates, \\[6pt]
M_{P}^{M_0}(x_{k},k)&:=& \sum_{x_{k-1} \in \mystates} M_{P}^{M_0}(x_{k-1},k-1) P(x_{k-1},x_{k},k-1)  \hspace{0.3cm} \forall \; k=1,\hdots,N, \; \; x_k \in \mystates.
\end{array}
\end{equation}
Now, let  $c : \mystates \times \mystates \times \mes(\mystates) \times \mes(\mystates) \to \R$,  $g : \mystates \times \mes(\mystates) \to \R$, $M: \time \to \mes(\mystates)$ and define $J_M:\mathcal{K}_{\mystates,\time}\to \R$   as 
$$
J_M(P):= \sum_{k=0}^{N-1} \sum_{\pos,\posn\in \mystates} M_{P}^{M_0} (\pos,k ) P(\pos, \posn , k) c_{\pos\posn} ( P(\pos, k) , M(k)) + \sum_{\pos\in \mystates} M_{P}^{M_0} (\pos, N) g(\pos , M (N)),
$$
where, for notational convenience, we have set $c_{\pos\posn}(\cdot, \cdot):= c(x,y, \cdot,\cdot)$ and $P(x,k):= P(x,\cdot,k) \in \mes(\mystates)$. 
We consider the following MFG problem: find $\hat{P} \in \mathcal{K}_{\mystates,\time}$ such that 
\begin{equation}\label{finite_MFG}
\hat{P} \in \mbox{argmin}_{P \in  \mathcal{K}_{\mystates,\time}} \; J_{M}(P) \; \; \; \mbox{with } \;  M= M_{\hat{P}}^{M_0}. \tag{MFG$_d$}
\end{equation}
In order to rewrite \eqref{finite_MFG} in a recursive form (as in \cite{GomeMohrSouza}), given $k =0,\hdots, N-1$, $x \in \mystates$ and $P \in  \mathcal{K}_{\mystates,\time}$, we define a probability distribution in $\mystates^{N-k+1}$ whose marginals are given by
$$\begin{array}{rcl}
M_{P}^{x,k}(x_{k},k)&:=& \delta_{x,x_k}, \hspace{0.3cm} \forall \;  x_k \in \mystates, \\[6pt]
M_{P}^{x,k}(x_{k'},k')&:=& \sum_{x_{k'-1} \in \mystates} M_{P}^{x,k}(x_{k'-1},k'-1) P(x_{k'-1},x_{k'},k'-1)  \hspace{0.3cm} \forall \; k'=k+1,\hdots,N, \; \; x_{k'} \in \mystates,
\end{array}
$$
where $\delta_{x,x_k}:=1$ if $x=x_k$ and $\delta_{x,x_k}:=0$, otherwise. Given $M: \time \to \mes(\mystates)$,  we also set 
$$\begin{array}{rcl}
J_M^{x,k}(P)&:=& \sum_{k'=k}^{N-1} \sum_{\pos,\posn\in \mystates} M_{P}^{x,k}(x_{k'},k') P(\pos, \posn , k') c_{\pos\posn} ( P(\pos, k') , M(k')) + \sum_{\pos\in \mystates} M_{P}^{x,k} (\pos, N) g(\pos , M (N))   \\[6pt]
\;   & = & \sum_{y \in \mystates} P(x,y,k) \left( c_{xy}( P(x,k), M(k)) + J_M^{y,k+1}(P)\right).
\end{array}
$$
Since for every $M: \time \to \mes(\mystates)$ the function 
$$
U_{M}(x,k):= \inf_{P \in  \mathcal{K}_{\mystates,\time}}J_M^{x,k}(P) \hspace{0.3cm} \forall \; k=0, \hdots, N-1, \; x\in \mystates,  \hspace{0.5cm} U_{M}(x,N):=g(\pos , M (N)), \; \; \forall \; x\in \mystates,
$$
satisfies the Dynamic Programming Principle (DPP),
\begin{equation}\label{DPP_formula_for_UM}
U_M(\pos,k) = \inf_{p \in \mes(\mystates)} \sum_{\posn\in \mystates} p(y) \left[ c_{\pos\posn} ( p , M(k)) + U_M(\posn , k+1) \right],\hspace{0.3cm} \forall \; k=0, \hdots, N-1, \; x\in \mystates,
\end{equation}
problem \eqref{finite_MFG} is equivalent to find $U: \mystates \times \time \to \R$ and $M: \time \to \mes(\mystates)$ such that 
\begin{equation}\label{introsFiniteMFGNashE}
\begin{split}
{\rm(i)}\quad& \; U(\pos,k) =  \sum_{\posn\in \mystates} \hat{P}(\pos, y, k )  \left[c_{\pos\posn} (\hat{P}(\pos, k ) , M(k) ) + U(\posn , k+1) \right] ,  \hspace{0.3cm} \forall \; k=0, \hdots, N-1, \; \; \; x\in \mystates,\\ 
{\rm(ii)}\quad&M(\pos, k) = \sum_{\posn \in \mystates } M(y, k-1)\hat{P}(\posn , \pos , k-1), \hspace{0.3cm} \forall \; k=1, \hdots, N, \; \; \; x\in \mystates, \\
{\rm(iii)}\quad&U(\pos, N) = g(\pos , N), \hspace{0.5cm} M(x,0) = M_0(x) \hspace{0.5cm} \forall \; x \in \mystates,  
\end{split}
\end{equation}
where $\hat{P}\in  \mathcal{K}_{\mystates,\time}$ satisfies 
\begin{equation}\label{kernel_minimizer_one_step_DP}
\hat{P}(\pos, \cdot , k) \in \argmin_{p \in \mes(\mystates)} \sum_{\posn\in \mystates} p(y) \left[ c_{\pos\posn} ( p , M(k) ) + U(\posn , k+1) \right] , \hspace{0.3cm} \forall \; k=0, \hdots, N-1, \; x\in \mystates.
\end{equation}
 As in \cite{GomeMohrSouza}, we will assume that \medskip\\
{\bf(H1)}  The following properties hold true:
\begin{itemize}
\item[{\rm(i)}] For every $x \in \mystates$ the functions $g(x,\cdot)$  and $\mes(\mystates)\times \mes(\mystates)  \ni (p,M)  \mapsto \sum_{\posn\in \mystates} p(y)c_{\pos\posn} ( p , M)$ are continuous.
\item[{\rm(ii)}] For every $U: \mystates \to \R$, $M \in \mes(\mystates)$ and  $\pos\in \mystates$, the optimization problem  
\begin{equation}\label{dynamic_optimization_with_a_unique_solution}
\inf_{p \in \mes(\mystates)} \sum_{\posn\in \mystates} p(y) \left[ c_{\pos\posn} ( p , M) + U(y)\right],
\end{equation}
admits a unique solution $\hat{p}(x,\cdot)  \in \mes(\mystates)$.
\end{itemize}
\begin{remark}\label{consequences_assumption_section_2}{\rm(i)} By using Brower's fixed point theorem, it is proved in {\rm\cite[Theorem 5]{GomeMohrSouza}} that under {\bf(H1)}, problem \eqref{finite_MFG} admits at least one solution.\smallskip\\
{\rm(ii)} As a consequence of the DPP, we have that {\bf(H1)}{\rm(ii)} implies that for every $M: \time \to \mes(\mystates)$, problem 
$$
\inf_{P \in  \mathcal{K}_{\mystates,\time}} J_{M}(P)
$$
admits a unique solution. \smallskip\\
{\rm(iii)} An example running cost $c_{\pos\posn}$ satisfying that $\mes(\mystates)\times \mes(\mystates)  \ni (p,M)  \mapsto \sum_{\posn\in \mystates} p(y)c_{\pos\posn} ( p , M)$ is continuous and {\rm {\bf (H1)}(ii)} is given by
\begin{equation}\label{entropy_penalized_cost}
c_{\pos\posn} ( p , M) := K(\pos,\posn,M) + \epsilon \log (p(y))
\end{equation}
where $\epsilon >0$, $K(\pos , \posn , \cdot)$ is continuous for all $\pos$, $\posn \in \mystates$, with the convention that $0 \log 0=0$. This type of cost has been already considered in {\rm\cite{GomeMohrSouza}}, and, given $x\in \mystates$,  the unique solution  of \eqref{dynamic_optimization_with_a_unique_solution} is given by
\begin{equation}\label{optimal_transition_probability} \hat{p}(\pos,y)=    \frac{\exp \left( -\left[K(\pos,\posn,M) + U(y)\right]/\epsilon \right)}{\sum_{\posn' \in \mystates} \exp \left(- \left[K(\pos,\posn',M) + U(y')\right]/\epsilon\right)}.
\end{equation}
In Section \ref{FirstorderMFGaslimitsoffiniteMFG} we will consider this type of cost in order to approximate continuous MFGs by finite ones. 
\end{remark} 
%
%
%

\section{Fictitious play for the finite MFG system}\label{Study_of_the_discrete_system}

Inspired by the fictitious play procedure introduced for continuous MFGs in \cite{hadikhanloo2017learning}, we consider in this section the convergence problem for the sequence of functions transition kernels $P_n \in \transition$ and marginal distributions $M_n: \time \to \mes(\mystates)$ constructed as follows: given $M_1 : \time \to \mes(\mystates)$ arbitrary, set $\bar{M}_1= M_1$ and,  for $n\geq 1$,  define 
\begin{equation}\label{fictitous_play_recurrence}
\begin{array}{rcl}
P_n&:=& \mbox{argmin}_{P \in  \mathcal{K}_{\mystates,\time}} \; J_{\bar{M}_n}(P),\\[6pt]
 M_{n+1}(\cdot,k)&:=& M_{P_n}^{M_0}(\cdot,k) , \hspace{0.6cm} \forall \; k=0,\hdots,N,  \\[6pt]
 \bar{M}_{n+1}(\cdot,k)&:=&  \frac{n}{n+1} \bar{M}_{n}(\cdot, k) + \frac{1}{n+1} M_{n+1}(\cdot, k), \hspace{0.6cm} \forall \; k=0,\hdots,N,
\end{array}
\end{equation}
where we recall that $M_0$ is given and for $P\in \transition$, the function $M_{P}^{M_0}: \mystates \times \time \to [0,1] $ is defined by \eqref{definition_M_P} (or recursively by \eqref{definition_M_P_recursive}). Note that by Remark \ref{consequences_assumption_section_2}{\rm(ii)}, the sequences $(P_n)$ and $(M_n)$ are well defined under {\bf(H1)}. 

The main object of this section is to show that, under suitable conditions, the sequence $(P_n)$ converges to a solution $\hat{P}$ to \eqref{finite_MFG} and $(M_n)$ converges to $M_{\hat{P}}^{M_0}$, i.e. the marginal distributions at the equilibrium. In practice, in order to compute $M_{n+1}$ from $\bar{M}_{n}$, we find first $P_n$ backwards in time by using the DPP expression for $U_{\bar{M}_n}$ in \eqref{DPP_formula_for_UM} and then we compute $M_{n+1}$ forward in time by using  \eqref{definition_M_P_recursive}. Notice that both computations are explicit in time. 
%

\subsection{Generalized fictitious play}  For the sake of simplicity, we present here an abstract framework that will allow us to prove the convergence of the sequence constructed in \eqref{fictitous_play_recurrence}. We begin by introducing some notations that will be also used in Section \ref{FirstorderMFGaslimitsoffiniteMFG}.
%
Let $\X$ and $\mathcal{Y}$  be two Polish spaces and $\Psi: \X  \to \mathcal{Y}$  be a Borel measurable function. Given a Borel probability measure $\mu$ on $\X$, we denote by $\Psi \sharp \mu$ the probability measure on $\mathcal{Y}$ defined by $\Psi \sharp \mu (A):= \mu(\Psi^{-1}(A))$ for all $A \in \B(\mathcal{Y})$.  Denoting by $\mes(\X)$ the set of Borel probability measures on $\X$ and by $d$ the metric on $\X$, we set  $\mes_{p}(\X)$ for the subset of $\mes(\X)$ consisting on measures $\mu$ such that $\int_{\X}d(x,x_0)^p \d \mu(x) < + \infty$ for some $x_0\in \X$. For $\mu_1$, $\mu_2 \in \mes_{p}(\X)$ define
$$\Pi (\mu_1 , \mu_2) := \{ \; \gamma \in \mes(\X \times \X) \; \big| \; \rho \sharp \pi_{1} = \mu_1 \hspace{0.2cm} \mbox{and } \hspace{0.2cm} \rho \sharp \pi_{2} = \mu_2\},$$
where $\pi_{1}$, $\pi_{2}: \X \times \X \to \R$, are defined by $\pi_{i}(x_1,x_2):= x_i$ for $i=1$, $2$. Endowed with the Monge-Kantorovic metric 
 $$\d_p (\mu_1 , \mu_2) = \inf_{ \gamma \in \Pi (\mu_1 , \mu_2) } \left( \int_{\X \times \X} \d  (x,y)^p \; \d \gamma(x,y) \right)^{1/p},$$
the set $\mes_p(\X)$ is shown to be a Polish space (see e.g. \cite[Proposition 7.1.5]{AGS}). Let us recall that $d_1$ corresponds to the Kantorovic-Rubinstein metric, i.e. 
\begin{equation}\label{distance_1_difference_lipschitz}
d_{1}(\mu_1, \mu_2)=  \sup\left\{ \int_{\X} f (x) \d (\mu_1-\mu_2)(x) \; ; \; f \in \mbox{Lip}_{1}(\R^d)\right\},
\end{equation}
where $ \mbox{Lip}_{1}(\X)$ denotes the set of Lipschitz functions defined in $\X$ with Lipschitz constant less or equal than $1$ (see e.g. \cite{Villani03}).   

 Let $\convexset \subseteq \X$ be a compact set. Then, by definition,  $\convexmset=\mes_p(\mathcal{C})$ for all $p \geq 1$, and $d_p$ metricizes the weak convergence of probability measures on $\mathcal{C}$ (see e.g. \cite[Proposition 7.1.5]{AGS}). Moreover, the set $\mes(\mathcal{C})$ is compact. 

Now, let $F:\convexset \times \convexmset \to \R$ be a given continuous function. Given $\varinC_{1}\in \mathcal{C}$ set $\bar{\eta}_1:= \delta_{x_1}$, the Dirac mass at $x_1$, and  for $n\geq 1$ define:
\begin{equation}\label{GenFP}
\varinC_{n+1} \in \argmin_{\varinC \in \convexset} F(\varinC,\bar \e_n), \quad \bar \e_{n+1} = \frac{1}{n+1} \sum_{k=1}^{n+1} \delta_{\varinC_{k}} = \frac{n}{n+1} \bar \e_n + \frac{1}{n+1}\delta_{\varinC_{n+1}}.
\end{equation}
We consider now the convergence problem of  the sequence $(\bar \e_{n})$  to some $\tilde{\e}\in \convexmset$ satisfying that
\begin{equation}\label{GenNE}
\mbox{supp}(\tilde{\eta}) \subseteq \argmin_{x\in \convexset} F(x,\tilde{\e}),
\end{equation}
where $\mbox{supp}(\tilde{\eta})$ denotes the support of the measure $\tilde{\eta}$. We call  such  $\tilde{\e}$  an equilibrium and its existence can be easily proved  by using Fan's fixed point theorem. 

We will prove the convergence of $(\tilde{\eta}_n)$ under a monotonicity and unique minimizer condition  for $F$.

\begin{definition}[Monotonicity]\label{GenMonotonicity} 
The function $F$ is called monotone, if  
\begin{equation}\label{montonictyinequality}
\int_{\convexset} \left( F(\varinC,\mu_1) - F(\varinC,\mu_2) \right) \; \d (\mu_1-\mu_2)(\varinC) \geq 0, \hspace{0.5cm} \forall \; \mu_1,  \; \mu_2 \in \mes(\mathcal{C}), \; \; \; \mu_1 \neq \mu_2.
\end{equation}
Moreover, $F$ is called strictly monotone if the inequality in \eqref{montonictyinequality} is strict. 
\end{definition}
\begin{definition}[Unique minimizer condition]\label{GenUniqueminimizer} The function $F$ satisfies the unique minimizer condition if for every $\e \in \convexmset$ the optimization problem $\inf_{\varinC \in \convexset} F(\varinC,\e)$ admits a unique solution.
\end{definition}
The following remark states some elementary consequence of the previous definitions.
\begin{remark}\label{GenUniquenessofEquilibrium}
{\rm(i)} If the unique minimizer condition holds then any equilibrium must be a Dirac mass. Moreover, the application $\mes(\mathcal{C}) \ni \eta \mapsto x_\eta:= \mbox{{\rm argmin}}_{x\in \mathcal{C}} F(x,\eta) \in \mathcal{C}$ is well defined and uniformly continuous. \\
{\rm(ii)} If $F$ is  monotone and the unique minimizer condition holds  then the equilibrium must be unique. Indeed, suppose that there are two different equilibria $\tilde{\e} = \delta_{\tilde{\varinC}}$ and $\tilde{\e}' = \delta_{\tilde{\varinC}'}$. Then, by the unique minimizer condition, 
$$
F(\tilde{\varinC} , \delta_{\tilde{\varinC}} ) <
F(\tilde{\varinC}' , \delta_{\tilde{\varinC}} ), \quad \mbox{and } \; \; \; 
F(\tilde{\varinC}' , \delta_{\tilde{\varinC}'} ) <
F(\tilde{\varinC} , \delta_{\tilde{\varinC}'} ).
$$
This gives $\int_{\convexset} \left( F(\varinC,\delta_{\tilde{\varinC}}) - F(\varinC,\delta_{\tilde{\varinC}'}) \right) \; \d (\delta_{\tilde{\varinC}}-\delta_{\tilde{\varinC}'})(\varinC) < 0,
$ which contradicts the monotonicity assumption. 

Arguing as in {\rm\cite[Proposition 2.9]{CDnotesonMFG})}, it is easy to see that uniqueness of the equilibrium also holds if $F$ is strictly monotone but does not necessarily satisfy the unique minimizer condition. 
\end{remark}

\begin{theorem}\label{Generalisedfictitiousplay}
Assume that
\begin{itemize}
\item[{\rm(i)}] $F$ is monotone and satisfies the unique minimizer condition.
\item[{\rm(ii)}] $F$ is Lipschitz, when $\mes(\mathcal{C})$ is endowed with the distance $\d_1$, and there exists $C>0$ such that 
\begin{equation}\label{change_of_points_in_F_estimate}\abs F(\varinC_1 , \e_1) - F(\varinC_1 , \e_2) - F(\varinC_2 , \e_1) + F(\varinC_2 , \e_2) \abs \leq C \left| \varinC_1 - \varinC_2 \right| \d_1 (\e_1,\e_2),\end{equation}
for all $x_1$, $x_2 \in \mathcal{C}$, and  $\mu_1$, $\mu_2 \in \mes(\convexset)$
\end{itemize}
Then, there exists $\tilde{\varinC} \in \convexset$ such that $\tilde{\e} =\delta_{\tilde{\varinC}}$ is the unique equilibrium and  the sequence $(\varinC_n , \bar \e_n)$ defined by  \eqref{GenFP} converges to $(\tilde{\varinC}, \delta_{\tilde{\varinC}})$.
\end{theorem}
Before we prove the theorem, let us recall a preliminary result (see \cite{hadikhanloo2017learning}).
\begin{lemma}\label{LEMFICT}
Consider a sequence of real numbers $(\phi_n)$ such that $\liminf_{n\to \infty} \phi_n \geq 0$. If there exists a real sequence $(\epsilon_n)$ such that $\lim_{n \to \infty} \epsilon_n = 0$ and 
$$\phi_{n+1} - \phi_n \leq -\frac{1}{n+1} \phi_n + \frac{\epsilon_n}{n}, \hspace{0.5cm} \forall \; n \in \N,$$
then $\lim_{n \to \infty} \phi_n = 0$.
\end{lemma}
\begin{proof}
Let $b_n = n \phi_n$ for every $n \in \N$. We have 
$$\frac{b_{n+1}}{n+1} - \frac{b_{n}}{n} \leq - \frac{b_{n}}{n(n+1)} + \frac{\epsilon_n}{n}, \hspace{0.5cm} \forall \; n \in \N,$$
which implies that $b_{n+1} \leq b_n + (n+1) \epsilon_n /n \leq b_n + 2 \abs \epsilon_n \abs$. Then, we get $b_n \leq b_1 + 2\sum_{i=1}^{n-1} \abs \epsilon_i \abs $ and, hence,  
$$0 \leq \liminf_{n\to \infty} \phi_n \leq \limsup_{n\to \infty} \phi_n \leq \lim_{n\to \infty} \frac{b_1 + 2 \sum_{i=1}^{n-1} \abs \epsilon_i \abs }{n} = 0,$$
from which the result follows.
\end{proof}
\begin{proof}[Proof of Theorem \ref{Generalisedfictitiousplay}.]
Let us define the real sequence $(\phi_n)$ as
$$\phi_n := \int_{\convexset} F(\varinC,\bar \e_n)  \d \bar \e_n(\varinC) - F(\varinC_{n+1},\bar \e_n).$$
We claim that $\phi_n \to 0$. Assuming that the claim is true, then any limit point $(\tilde{x}, \tilde{\eta})$ of $(x_{n+1}, \bar \e_n)$ satisfies  
$$
F(\tilde{x}, \tilde{\eta}) \leq F(x, \tilde{\eta}) \hspace{0.3cm} \forall \; x\in \mathcal{C}, \hspace{0.4cm} \mbox{and } F(\tilde{x}, \tilde{\eta}) = \int_{\convexset} F(\varinC,\tilde{\eta})  \d \tilde{\eta}(\varinC),
$$
which implies that $\tilde{\eta}$ satisfies \eqref{GenNE}, i.e. $\tilde{\eta}$ is an equilibrium. Using that $F$ is monotone and Remark \ref{GenUniquenessofEquilibrium}{\rm(ii)}, the assertions on the theorem follows. 

Thus, it remains to show that $\phi_n \to 0$, which will be proved with the help of Lemma \ref{LEMFICT}. By definition of $x_{n+1}$ we have that $\phi_n \geq 0$. Let us write $\phi_{n+1} - \phi_n = A+B$, where
$$A=\int_{\convexset} F(\varinC,\bar \e_{n+1}) \; \d \bar \e_{n+1}(\varinC) - \int_{\convexset} F(\varinC,\bar \e_n) \; \d \bar \e_n(\varinC), \quad B =  F(\varinC_{n+1},\bar \e_n) - F(\varinC_{n+2},\bar \e_{n+1}).$$
We have
\begin{equation}\label{estimate_for_B}
\begin{split}
B &\leq F(\varinC_{n+2},\bar \e_n) - F(\varinC_{n+2},\bar \e_{n+1}) \\[6pt]
&\leq F(\varinC_{n+1},\bar \e_n) - F(\varinC_{n+1},\bar \e_{n+1}) + C | \varinC_{n+2} - \varinC_{n+1} | \d_1 (\bar \e_n , \bar \e_{n+1}) \\
&\leq F(\varinC_{n+1},\bar \e_n) - F(\varinC_{n+1},\bar \e_{n+1}) + \frac{C}{n+1} | \varinC_{n+2} - \varinC_{n+1} | \d_1 (\delta_{x_{n+1}} , \bar \e_{n}) ,
\end{split}
\end{equation}
where we have used  \eqref{change_of_points_in_F_estimate} to pass from the first to the second inequality and \eqref{distance_1_difference_lipschitz} from the second to the third inequality. Similarly, using \eqref{GenFP} and that $F$ is Lipschitz,
\begin{equation}\label{estimate_for_A}
\begin{split}
A&= \int_{\convexset} (F(\varinC,\bar \e_{n+1}) - F(\varinC,\bar \e_n)) \; \d \bar \e_n(\varinC) + \frac{1}{n+1} \left[F(\varinC_{n+1},\bar \e_{n+1}) - \int_{\convexset} F(\varinC,\bar \e_{n+1}) \; \d \bar \e_{n}(\varinC) \right] \\
&\leq \int_{\convexset} (F(\varinC,\bar \e_{n+1}) - F(\varinC,\bar \e_n)) \; \d \bar \e_n(\varinC)  + \frac{1}{n+1} \left[F(\varinC_{n+1},\bar \e_{n}) - \int_{\convexset} F(\varinC,\bar \e_{n}) \; \d \bar \e_{n}(\varinC) \right] +  \frac{C}{n+1} \d_1 (\bar \e_n , \bar \e_{n+1}) \\
&\leq \int_{\convexset} (F(\varinC,\bar \e_{n+1}) - F(\varinC,\bar \e_n)) \; \d \bar \e_n(\varinC) - \frac{1}{n+1} \phi_{n} + \frac{C}{(n+1)^2} \d_1 (\bar \e_n , \delta_{x_{n+1}}). \\
\end{split}
\end{equation}
On the other hand, the second relation in \eqref{GenFP} yields $-(n+1)( \bar{\eta}_{n+1}-\bar{\eta}_{n})=  \bar{\eta}_{n}-\delta_{x_{n+1}}$. Therefore, 
\begin{equation}\label{monotonicity_inside_the_proof}
\begin{split}
F(\varinC_{n+1},\bar \e_n) - F(\varinC_{n+1},\bar \e_{n+1})  + \int_{\convexset} (F(\varinC,\bar \e_{n+1}) - F(\varinC,\bar \e_n)) \; \d \bar \e_{n}(\varinC) = \\
-(n+1) \int_{\convexset} (F(\varinC,\bar \e_{n+1}) - F(\varinC,\bar \e_n)) \; \d (\bar \e_{n+1} - \bar \e_{n})(\varinC) \leq 0,
\end{split}
\end{equation}
by the monotonicity condition of $F$. From estimates \eqref{estimate_for_B}-\eqref{estimate_for_A} and inequality \eqref{monotonicity_inside_the_proof}  we deduce that
\begin{equation}
\begin{split}
\phi_{n+1} - \phi_{n} &\leq - \frac{1}{n+1} \phi_{n} + \frac{C}{n+1} \d_1 (\delta_{x_{n+1}}, \bar \e_{n}) \left( \frac{1}{n+1} + | \varinC_{n+2} - \varinC_{n+1} | \right).
\end{split}
\end{equation}
Using that $\mes(\mathcal{C})$ is compact (and so bounded in $\d_1$), we get that  
$$
\phi_{n+1} - \phi_{n} \leq - \frac{1}{n+1} \phi_{n} +\frac{\epsilon_n}{n},
$$
where $\epsilon_n := C' (\frac{1}{n+1} + | \varinC_{n+2} - \varinC_{n+1} |)$, with $C'>0$ and  independent of $n$.  Remark \ref{GenUniquenessofEquilibrium} implies that $| \varinC_{n+2} - \varinC_{n+1} | \to 0$ as $n\to \infty$ (because $\d_1 (\bar \e_{n},\bar{\e}_{n+1})= \d_1 (\bar \e_{n},\delta_{x_{n+1}})/(n+1) \to 0$). Thus, $\epsilon_{n} \to 0$ and the result follows from Lemma \ref{LEMFICT}.
\end{proof}

\subsection{Convergence of the fictitious play for finite MFG}
In this section, we apply the abstract result in 	Theorem \ref{Generalisedfictitiousplay} to the finite MFG problem \eqref{finite_MFG}. Under the notations of Section \ref{preliminaries}, in what follows,  will assume that $c_{xy}(\cdot, \cdot)$ has a separable form. Namely, 
\begin{equation}\label{SeperableCostFunctionFiniteMFG}
c_{\pos\posn} ( p , M) = K(\pos,\posn,p) + f(\pos,M), \hspace{0.4cm} \forall \; x, \; y \in \mystates, \; \; p, \; M \in \mes(\mystates),
\end{equation} 
where $K : \mystates \times \mystates \times \mes(\mystates)  \to \R$ and $f: \mystates \times \mes(\mystates) \to \R$ are given. 
In order to write \eqref{finite_MFG} as a particular instance of \eqref{GenNE}, given $\eta \in \mes(\mathcal{K}_{\mystates,\time})$ we define $M_{\e}:= \time \to \mes(\mystates)$ and $F: \mathcal{K}_{\mystates,\time} \times \mes(\mathcal{K}_{\mystates,\time}) \to \R$ as 
\begin{equation}\label{GenFormofCostFunction}
M_{\e}(k) := \int_{\mathcal{K}_{\mystates,\time}} M_{P}^{M_0} (k) \;\d \e(P),  \hspace{0.3cm} \forall \; k=0, \hdots, N, \; \; \; \; \mbox{and } \hspace{0.4cm}  F(P,\e) := J_{M_\e}(P).
\end{equation}

Under assumption {\bf(H1)}, we have that $F$ is continuous and satisfies the unique minimizer condition in Definition \ref{GenUniqueminimizer}. Therefore, by Remark \ref{GenUniquenessofEquilibrium}{\rm(i)}, associated to any equilibrium $\eta \in \mes(\mathcal{K}_{\mystates,\time})$ for $F$, i.e. $\eta$ satisfies \eqref{GenNE} with $\mathcal{C}= \mathcal{K}_{\mystates,\time}$,  there exists  $P_\eta \in \mathcal{K}_{\mystates,\time}$ such that $\eta= \delta_{P_\eta}$, from which we get that $P_\eta$ solves \eqref{finite_MFG}. Conversely, for any solution $P$ to \eqref{finite_MFG} we can associate the measure $\eta_{P}:= \delta_{P}$, which solves  \eqref{GenFP}. An analogous argument shows that the fictitious play procedures \eqref{fictitous_play_recurrence} and \eqref{GenFP} are equivalent. 

We consider now some assumptions on the data of the finite MFG problem that will ensure the validity of assumptions {\rm(i)}-{\rm(ii)} for $F$ in Theorem \ref{Generalisedfictitiousplay}. \medskip\\
{\bf(H2)} We assume that 
\begin{itemize}
\item[{\rm(i)}] $f$ and $g$ are monotone, in the sense that setting $h=f$, $g$, we have
$$
\sum_{x \in \mystates} \left(h(x,M) -h(x,M')\right)(M(x) - M'(x)) \geq 0 \hspace{0.5cm} \forall \; M, \; M' \in \mes(\mystates). 
$$
\item[{\rm(ii)}] $f$ and $g$ are Lipschitz with respect to their second argument. 
\end{itemize}

The following result is a straightforward consequence of the definitions. 
%
\begin{lemma}\label{monotonicity_finite_mfg}
If $f$ and $g$ are monotone, then  $F$ is monotone  in sense of Definition \ref{GenMonotonicity}.
\end{lemma}
\begin{proof}
For any two distributions $\e,\e' \in \convexmset$ we want to show $
\int_{\convexset} \left( F (P,\eta) - F (P,\eta') \right) \; \d (\eta - \eta')(P) \geq 0.
$
By using the exact form of the cost function $F$ by equation \eqref{GenFormofCostFunction} and taking into account the separable form of the running cost \eqref{SeperableCostFunctionFiniteMFG}, we have:
\begin{equation*}
\begin{split}
F (P,\eta) - F (P,\eta') =\sum_{k=0}^{N-1} &\sum_{\pos\in \mystates} M_{P}^{M_0} (\pos,k) \left[f( \pos , M_\e (k )) - f( \pos , M_{\e'} (k ) )\right]
\\
+
&\sum_{\pos\in \mystates} M_{P}^{M_0} (\pos, N) \left[g( \pos , M_\e (N)) - g( \pos , M_{\e'} (N) )\right].
\end{split}
\end{equation*}
Thus, \small
\begin{equation*}
\begin{split}
\int_{\mathcal{K}_{\mystates,\time}} \left( F (P,\eta) - F (P,\eta') \right) \; \d (\eta - \eta')(P)
= 
\sum_{k=0}^{N-1} &\sum_{\pos\in \mystates} \left[f( \pos , M_\e (k)) - f( \pos , M_{\e'} (k) )\right]
\int_{\mathcal{K}_{\mystates,\time}} M_{P}^{M_0}(\pos,k) \; \d (\eta - \eta')(P) \\
+
&\sum_{\pos\in \mystates}  \left[g( \pos , M_\e (N)) - g( \pos , M_{\e'} (N) )\right]
\int_{\mathcal{K}_{\mystates,\time}} M_{P}^{M_0}(\pos,N) \; \d (\eta - \eta')(P) \\
= 
\sum_{k=0}^{N-1} &\sum_{\pos\in \mystates} \left[f( \pos , M_{\e}(k)) - f( \pos , M_{\e'}(k) )\right]
(M_{\e}(\pos , k) - M_{\e'}(\pos , k) ) \\
+
&\sum_{\pos\in \mystates} \left[g( \pos , M_{\e}(N)) - g( \pos , M_{\e'}(N) )\right]
(M_{\e}(\pos , N ) - M_{\e'}(\pos, N) ) \geq 0,
\end{split}
\end{equation*}\normalsize
where the inequality above follows from from the monotonicity of $f$ and $g$.
\end{proof}
By Remark \ref{GenUniquenessofEquilibrium} we directly deduce the following result.
\begin{proposition}
If {\bf(H1)} and {\bf(H2){\rm(i)}}  hold, then the finite MFG \eqref{finite_MFG} has a unique equilibrium.
\end{proposition}
\begin{remark} The previous result slightly improves {\rm\cite[Theorem 6]{GomeMohrSouza}}, where the uniqueness of the equilibrium is proved under a stronger strict monotonicity assumption on $f$ and $g$. 
\end{remark}
In order to check assumption {\rm(ii)} in Theorem \ref{Generalisedfictitiousplay}, we need first a preliminary result.

\begin{lemma}\label{LipschitzConditionforInducedMeasures}
There exists a constant $C>0$ such that
\begin{equation}\label{lipschitz_1_MP}| M_{P}^{M_0}(k) - M_{P'}^{M_0}(k) | \leq  C | P - P' |_{\infty} \hspace{0.3cm} \forall \; P, \; P' \in \mathcal{K}_{\mystates , \time}, \; \; k=0,\hdots, N.\end{equation}
In particular,
\begin{equation}\label{lipschitz_2_MP}| M_{\e} (k)  - M_{\e'} (k) | \leq C  d_1 (\e , \e') \hspace{0.3cm}  \forall \; \e, \; \e' \in \mes(\mathcal{K}_{\mystates , \time}),  \; \; k=0,\hdots, N.\end{equation}
\end{lemma}
\begin{proof}
For any $k=0, \hdots, N-1$ and $x\in \mystates$ we have
\begin{equation}\label{3}
\begin{split}
M_{P}^{M_0}(\pos, k+1) - M_{P'}^{M_0}(\pos,k+1)
&= \sum_{\posn \in \mystates} M_{P}^{M_0} (\posn,k) P(\posn, \pos, k) - \sum_{\posn \in \mystates} M_{P'}^{M_0} (\posn,k) P'(\posn, \pos, k) \\
&\leq \sum_{\posn \in \mystates} M_{P}^{M_0} (\posn,k) (P(\posn, \pos, k) - P'(\posn, \pos, k) )\\
& \hspace{0.4cm} +   |M_{P}^{M_0}(k) - M_{P'}^{M_0}(k)|_{\infty}\sum_{\posn \in \mystates}  P'(\posn, \pos, t_k) \\
&\leq     | P - P' |_{\infty}+ |\mystates||M_{P}^{M_0}(k) - M_{P'}^{M_0}(k)|_{\infty},
\end{split}
\end{equation}
where we have used that $\sum_{\posn \in \mystates } M_{P}^{M_0} (\posn,k)=1$.  Using that  $M_{P}^{M_0}(0) =M_{P'}^{M_0}(0)=M_0$, inequality \eqref{lipschitz_1_MP} follows by applying \eqref{3} recursively. Now, given  $\gamma \in \Pi(\e, \e')$, i.e. $\gamma \in \mes(\mathcal{K}_{\mystates , \time} \times \mathcal{K}_{\mystates , \time})$ with marginals given by $\e$ and $\e'$, we have
$$\begin{array}{rcl}
| M_{\e}(k)  - M_{\e'}(k) | &=& \left| \int_{\mathcal{K}_{\mystates , \time}} M_{P}^{M_0}(k) \; \d \e(P)  - \int_{\mathcal{K}_{\mystates , \time}} M_{P'}^{M_0}(k) \; \d \e' (P') \right|  \\[6pt] 
\; & =& \left| \int_{ \mathcal{K}_{\mystates , \time}\times \mathcal{K}_{\mystates , \time}}(M_{P}^{M_0}(k) - M_{P'}^{M_0}(k)) \; \d \gamma (P,P') \right|\\[6pt] 
&\leq& C \int_{\mathcal{K}_{\mystates , \time}\times \mathcal{K}_{\mystates , \time}} | P - P' |_{\infty} \; \d \gamma (P,P').
\end{array}$$
Inequality \eqref{lipschitz_2_MP} follows by taking the infimum over $\gamma \in \Pi(\e, \e')$.
\end{proof}
\begin{lemma}\label{stability_different_points_and_measures_finite_mfg}
Assume that {\bf(H2)}{\rm(ii)} holds. Then, there exists $C>0$ such that
\begin{equation}\label{LPCONDFPFMFG}
\begin{split}
\abs F(P,\e) - F(P,\e') - F(P',\e) + F(P',\e') \abs &\leq C \;    |P - P'|_{\infty} \d_1 (\e , \e'), \\
\abs F(P,\e) - F(P,\e') \abs &\leq C \d_1 (\e , \e'),
\end{split}
\end{equation}
for all $P$, $P' \in \mathcal{K}_{\mystates , \time}$ and $\e$, $\e' \in \mes(\mathcal{K}_{\mystates , \time})$.
\end{lemma}
\begin{proof} Let us first prove the second relation in \eqref{LPCONDFPFMFG}. By {\bf(H2)}{\rm(ii)} and Lemma \ref{LipschitzConditionforInducedMeasures} we can write  $\abs F(P,\eta) - F (P,\eta') \abs \leq A+B$ with
$$
A:=\sum_{k=0}^{N-1} \sum_{\pos\in \mystates} M_{P}^{M_0} (\pos,k)   \abs f( \pos , M_\e (k)) - f( \pos , M_{\e'} (k)) \abs \leq c\sum_{k=0}^{N-1} \sum_{\pos\in \mystates} M_{P}^{M_0} (\pos,k) \; \d_1 (\e , \e') = cN\d_1 (\e , \e'),$$
and
$$
B:= \sum_{\pos\in \mystates} M_{P}^{M_0} (\pos,N)   \abs g( \pos , M_\e (N)) - g( \pos , M_{\e'} (N)) \abs \leq c \sum_{\pos\in \mystates} M_{P}^{M_0} (\pos,N) \;  \d_1 (\e , \e') = c \d_1 (\e , \e'),
$$
for some $c>0$. Thus, the second estimate in \eqref{LPCONDFPFMFG} follows. In order to prove the first relation in \eqref{LPCONDFPFMFG}, let us write $| F(P,\e) - F(P,\e') - F(P',\e) + F(P',\e') | \leq A'+ B'$ with
$$A':= \sum_{k=0}^{N-1} \sum_{\pos\in \mystates} \abs M_{P} (\pos,k)  - M_{P'} (\pos, k) \abs \; \abs f( \pos , M_\e (k))) - f( \pos , M_{\e'} (k)) \abs \leq CN \abs \mystates \abs | P - P'|_{\infty}\d_1 (\e , \e'),$$
$$B':= \sum_{\pos\in \mystates} \abs M_{P} (\pos,N)  - M_{P'} (\pos,N) \abs \; \abs g( \pos , M_\e (N))) - g( \pos , M_{\e'} (N)) \abs \leq C \abs \mystates \abs | P - P' |_{\infty} \d_1 (\e , \e').$$
The result follows. 
\end{proof}
By combining Lemma \ref{monotonicity_finite_mfg}, Lemma \ref{stability_different_points_and_measures_finite_mfg} and Theorem \ref{Generalisedfictitiousplay}, we get the following convergence result. 
\begin{theorem}\label{finiteMFGFictitiousPlayTheorem}
Assume {\bf(H1)} and {\bf(H2)} and let $(P_n, M_{n}, \bar{M}_n)$ be the sequence generated in  the fictitious play procedure \eqref{fictitous_play_recurrence}. Then, $(P_n, M_{n}, \bar{M}_n)\to (\hat{P}, M^{M_0}_{\hat{P}},  M^{M_0}_{\hat{P}})$, where $\hat{P}$ is the unique solution to \eqref{finite_MFG}.
\end{theorem}

\section{First order MFG as limits of finite MFG}\label{FirstorderMFGaslimitsoffiniteMFG}
In this section we consider a  relaxed first order MFG problem  in continuous time and with a continuum of states.   We define a natural finite MFG associated to a discretization of the space and time variables. We address our second main question in this  work, which is the convergence of the solutions of  finite  MFGs to  solutions of  continuous MFGs when the discretization parameters tend to zero. 

In order to introduce the MFG problem, we need first to introduce some definitions. Let us define $\Gamma:= C ([0,T] ; \R^d)$ and given $m_0 \in \mes(\R^d)$, called the initial distribution, let 
$$\mes_{m_0}(\Gamma) = \left\{ \e \in \mes(\Gamma) \;  ; \; e_0 \sharp \e = m_0 \right\},$$
where, for each $t\in [0,T]$, the function $e_t: \Gamma \to \R^d$ is defined by $e_t(\gamma)=\gamma(t)$.  Let $\ell:  \R^{d} \to \R$  and $f$, $g: \R^d\times \mes_1(\R^d) \to \R$. Given $m \in C([0,T];\mes_1(\R^d))$ and $q \in (1,+\infty)$, we consider the following family of variational problems, parametrized by the initial condition, 
\begin{equation}\label{family_of_optimal_control_problems}
\inf \left\{ \int_{0}^{T} \left[ \ell(\dot{\gamma}(t))+ f(\gamma(t) ,m(t)) \right] \d t + g(\gamma(T) ,m(T)) \; \; \big| \; \; \gamma \in W^{1,q}([0,T]; \R^d), \; \; \gamma(0)= x \right\}, \hspace{0.3cm} x\in \R^d.
\end{equation}
\begin{definition}\label{definition_mfg_equilibrium} We call  $\xi^* \in \mes_{m_0}(\Gamma)$ a {\it {\rm MFG} equilibrium for \eqref{family_of_optimal_control_problems}}   if $[0,T] \ni t \mapsto  e_t \sharp \xi^*$ belongs to $C([0,T]; \mes_1(\R^d))$ and $\xi^*$-almost every $\y$ solves the optimal control problem in \eqref{family_of_optimal_control_problems} with $x=\gamma(0)$ and $m(t)=e_t \sharp \xi^*$ for all $t\in [0,T]$. 
\end{definition}

Assuming that the cost functional of the optimal control problem in \eqref{family_of_optimal_control_problems} is meaningful, which is ensured by the conditions on  $\ell$,  $f$ and $g$ in assumption {\bf(H3)} below, the interpretation of a MFG equilibrium  is as follows: the measure $\xi^\ast$ is an equilibrium if it only charges trajectories in $\R^d$, distributed as $m_0$ at the initial time,   minimizing a cost depending on the collection of time marginals  of $\xi^\ast$ in $[0,T]$.

\begin{remark} Usually, see e.g. {\rm\cite{LL07mf}} and {\rm\cite{CDnotesonMFG}}, a first order MFG equilibrium is presented in the form of a system of PDEs consisting in a HJB equation, modelling the fact that a typical agent solves an optimal control problem, which depends on the marginal distributions of the agents at each time $t\in [0,T]$, coupled with a continuity equation, describing the evolution of the aforementioned marginal distributions if the agents follow the optimal dynamics. The definition of equilibrium that we adopted in this work  corresponds  to a relaxation of the PDE notion of equilibrium, and has been used, for instance, in {\rm\cite{CDHD2015}}, {\rm\cite[Section 3]{MR3644590}} and, recently, in {\rm\cite{Cannarsa_Capuani_17}}.
\end{remark}
Throughout this section, we will suppose that the following assumption holds. \medskip\\
{\bf(H3)}{\rm(i)}  The function $\ell$ is continuous and there exist constants $\underline{\ell}>0$, $\overline{\ell}>0$ and $C_\ell >0$ such that 
\begin{equation}\label{polynomial_growth_ell}
 \underline{\ell}|\alpha|^{q} - C_\ell \leq \ell(\alpha) \leq \overline{\ell}|\alpha|^{q} + C_\ell \hspace{0.3cm} \forall \; \alpha \in \R^d.
\end{equation}
{\rm(ii)} For $h=f$, $g$ we have that $h$ is continuous, $h(\cdot, m)$ is $\C^1$, for every $m\in \mes_1(\R^d)$, and there exists $C>0$ such that
\begin{equation}\label{bounds_on_f_and_g} \sup_{m \in \mes_1 (\R^d) } \left\{   \norm h(\cdot , m)\norm_{\infty} + \norm D_x h(\cdot , m) \norm_{\infty} \right\} \leq C.
\end{equation} 
{\rm(iii)} The initial distribution $m_0\in \mes(\R^d)$  has a compact support. \smallskip\\


Now we will focus on a particular class of finite MFGs and relate their solutions, asymptotically,  with the MFG equilibria for \eqref{family_of_optimal_control_problems}. Let $\left(N_{n}^{s}\right)$ and $\left(N_{n}^{t}\right)$ be two sequences of natural numbers such that $\lim_{n\to \infty} N_{n}^{s}=\lim_{n\to \infty} N_{n}^{t}=+\infty$ and let $(\epsilon_n)$ be a sequence of positive real numbers such that $\lim_{n\to \infty} \epsilon_n=0$. Define $\Delta x_{n} := 1/N_{n}^{s}$ and $\dt_{n}:=T/N_{n}^t$. For a fixed $n\in \N$, consider  the discrete state set $\mystates_{n}$ and the discrete time set $\mathcal{T}_{n}$ defined as
\begin{equation}\label{definition_discrete_grids}
\begin{array}{rcl}
\mystates_{n} & :=& \left\{ x_{i}:=  i \Delta x_n \; \; \abs \; \; i \in \mathbb{Z}^{d}, \; \; \abs i \abs_{\infty} \leq (N_{n}^s)^2 \right\} \subseteq \R^d, \\[7pt]
\mathcal{T}_{n} &:= & \left\{ t_{k}:= k \Delta t_n \; \; \abs \; \; k=0, \hdots, N_{n}^t\right\} \subseteq [0,T].
\end{array}
\end{equation}
Let us also define the (non positive) entropy function
$\mathcal{E}_n: \mes(\mystates_n) \to  \R $ by
$$ \quad \mathcal{E}_n (p) = \sum_{\pos \in \mystates_n} p(x) \log (p(x)) \hspace{0.3cm} \forall \; p \in \mes(\mystates_n),$$
with the convention  that  $0\log 0=0$.  For every $x \in \mystates_n$ set $E_x^n:= \left\{ x' \in \R^d \; | \; |x'-x|_{\infty} \leq \Delta x_n/2\right\}$. Since we will be interested in the asymptotic as $n\to \infty$, we can assume, without loss of generality, that $m_0( \partial E_x^n)=0$ for all $x \in \mystates_n$. Similarly, by {\bf(H3)}{\rm(iii)}, we can assume that the support of $m_0$ will be contained in $\cup_{x\in \mystates_n} E_x^n$. Based on these considerations, setting
%
$$M_{n,0}(x):= m_0(E_x^n) \hspace{0.3cm} \forall \; x\in \mystates_n,$$
we have that $M_{n,0} \in  \mes (\mystates_n)$. 
We consider the finite MFG,  written in a recursive form (see \eqref{introsFiniteMFGNashE}),    
\begin{equation}\label{introsFiniteMFGNashEinConvQuestion}
\begin{array}{l}
{\rm(i)} \; \; U_n (\pos,t_k)=   \min_{p \in \mes(\mystates_n)} \left\{\sum_{\posn\in \mystates_n} p(y) \left[ \dt_{n}  \ell\left(\frac{ \posn - \pos}{\dt_{n}}\right)  + U_n ( \posn, t_{k+1}) \right]+  \epsilon_n \mathcal{E}_n (p)\right\}  \\[10pt]
  \hspace{2.3cm}+ \dt_{n} f(\pos,M_n(t_k) )  \hspace{0.5cm} \forall   \; \pos \in \mystates_n, \; \;  0 \leq k < N_n^{t}  ,\\[6pt]
{\rm(ii)} \; \;  M_n( \posn, t_{k+1}) =  \sum_{x \in \mystates_n }  \hat{P}_{n}(x,y,t_k)  M_n (x,t_k) \hspace{0.5cm} \forall \; y \in \mystates_n, \; \; 0 \leq k  < N_n^{t}   ,\\[10pt]
{\rm(iii)} \; \;  M_n(x,0)=   M_{n,0}(x), \hspace{0.5cm} U_n(\pos, T)= g(\pos , M_n(T)) \hspace{0.5cm} \forall \; \pos \in \mystates_n,
\end{array}
\end{equation}  
where for all  $x \in \mystates_{n}$,   $0 \leq k \leq  N_n^{t}-1$, $\hat{P}_n( \pos, \cdot  ,t_k ) \in \mes (\mystates_n)$ is given by
\begin{equation}\label{optimal_transition_kernel}
\hat{P}_n( \pos, \cdot  ,t_k ) = \argmin_{p \in \mes(\mystates_n)} \left\{ \sum_{  \posn\in \mystates_n} p(y) \left[ \dt_{n} \ell\left( \frac{ \posn - \pos}{\dt_{n}}\right)  + U_n (\posn , t_{k+1}) \right] + \epsilon_n \mathcal{E}_n (p)\right\},
\end{equation}
and, by notational convenience,  every $p \in  \mes(\mystates_n)$ is identified with $\sum_{x \in \mystates_n} p(x) \delta_{x} \in \mes_1(\R^d)$.
Note that system \eqref{introsFiniteMFGNashEinConvQuestion} is a particular case of \eqref{introsFiniteMFGNashE}, with
$$c_{\pos\posn} (p,M) := \dt_{n}\left[ \ell\left( \frac{ \posn - \pos}{\dt_{n}}\right) +  f (\pos , M)\right] + \epsilon_n \log(p(y)).$$
\begin{remark} The positive parameter $\epsilon_n$  and the entropy term $\mathcal{E}_n$  are introduced in \eqref{introsFiniteMFGNashEinConvQuestion} in order to ensure that $\hat{P}_n$ is well-defined, and so that assumption {\bf(H1)} for system \eqref{introsFiniteMFGNashEinConvQuestion} is satisfied in this case.  In particular, Remark \ref{consequences_assumption_section_2} ensures the existence of at least one solution $(U_n,M_n)$ of \eqref{introsFiniteMFGNashEinConvQuestion}, with associated transition kernel $\hat{P}_n$ given by \eqref{optimal_transition_kernel}. 
\end{remark}

 In order to study the asymptotic behaviour of $(U_n,M_n,\hat{P}_n)$, let us first introduce some useful notations. We set $\mathcal{K}_{n}:= \mathcal{K}_{\mystates_n, \time_n}$ (see Definition \ref{definition_transition_kernel}) and, given $x\in \mystates_n$ and $t\in \time_n$, we denote by $\Gamma_{x,t}^{\mystates_n,\time_n} \subseteq \Gamma_t$  the set of continuous functions $\gamma: [t,T] \to \R^d$ such that $\gamma(t)=x$ and, for each $1 \leq k \leq m$, with  $ t_k\in \time_n \cap (t,T]$, we have that $\gamma(t_k) \in \mystates_n$ and the restriction of $\gamma$ to the interval $[t_{k-1},t_{k}]$ is affine.  Given $P\in \mathcal{K}_n$ let us define $\xi_{P}^{x,t,n}\in \mes(\Gamma_t)$ by
\begin{equation}\label{dirac_masses_in_gamma}
\xi_{P}^{x,t,n}:= \sum_{\gamma \in \Gamma_{x,t}^{\mystates_n,\time_n}} p_{P}^{x,t,n}(\gamma) \delta_{\gamma}, \hspace{0.2cm} \mbox{where } \; \; 
p_{P}^{x,t,n}(\gamma):= \prod_{t_k \in \time_n \cap [t,T]} P(\gamma(t_k), \gamma(t_{k+1}),t_k).
\end{equation}
For a given Borel measurable function $L: \Gamma_t \to \R$ and $\xi \in \mes(\Gamma_t)$  we will denote $\E_{\xi}(L):= \int_{\Gamma_t} L(\gamma) \d \xi(\gamma)$, provided that the integral is well-defined.   Using these notations,  expression   \eqref{introsFiniteMFGNashEinConvQuestion}{\rm(i)} is equivalent to 
\begin{equation}\label{NEminimizer1} 
\begin{array}{l}
U_n(x,t_{k}) = \min_{P\in \mathcal{K}_n} \; \left\{\E_{\xi_{P}^{x,t_k,n}} \left(  \dt_{n}\sum_{k'=k}^{N_{n}^{t}-1} \left[  \ell\left(\frac{\y(t_{k'+1})- \y(t_{k'})}{\dt_n}\right)  + f(\y(t_{k'}),M_n(t_{k'})) \right] \right)\right. \\[10pt]
\hspace{3.2cm}\left.+  \E_{\xi_{P}^{x,t_k,n}} \left(g(\y(T) , M_n(T))\right)+\epsilon_n  \E_{\xi_{P}^{x,t_k,n}} \left(\sum_{k'=k}^{N_{n}^{t}-1} \log P(\gamma(t_{k'}),\gamma(t_{k'+1}),t_{k'})\right)\right\},
\end{array}
\end{equation}
for all $x\in \mystates_n$ and $k=0, \hdots, N_{n}^{t}-1$.  For latter use, note that since the support of $\xi_{P}^{x,t_k,n}$ is contained in $\Gamma_{x,t_k}^{\mystates_n,\time_n}$, for  $\xi_{P}^{x,t_k,n}$ almost every $\gamma \in \Gamma_t$ we have that $\dot{\gamma}(t)= (\gamma(t_{k'+1})-\gamma(t_{k'}))/\dt_n$ for every $k'= k, \hdots, N_{n}^{t}-1$ and $t \in (t_{k'},t_{k'+1})$, and, hence, 
\begin{equation}\label{quotient_is_derivative}
\E_{\xi_{P}^{x,t_k,n}} \left(  \dt_{n}\ell\left(\frac{\y(t_{k'+1})- \y(t_{k'})}{\dt_n}\right) \right)=\E_{\xi_{P}^{x,t_k,n}} \left(\int_{t_{k'}}^{t_{k'+1}} \ell\left(\dot{\y}(t) \right) \d t \right).
\end{equation}
Finally, let us define  $\xi_{n} \in \mes(\Gamma)$ by
\begin{equation}\label{definition_xi_n}\xi_{n} := \sum_{\pos \in \mystates} M_{n,0}(x) \; \xi_{\hat{P}_n}^{x,0,n}.\end{equation}
Notice that, by definition, $M_{n}(t)= e_{t} \sharp \xi_n$ for all $t\in \time_n$. We extend $M_{n}: \time_n \to \mes_1(\R^d)$ to  $M_{n}: [0,T] \to \mes_1(\R^d)$ via the formula
\begin{equation}\label{interpolations}
M_{n}(t):= e_t \sharp \xi_n \hspace{0.4cm} \mbox{for all } t\in [0,T]. 
\end{equation}

\subsection{Convergence analysis}
We now study the limit behaviour of the solutions $(U_n, M_n)$ in \eqref{introsFiniteMFGNashEinConvQuestion}, and of the associated sequence $(\xi_n)$,  as $n \to \infty$.  
We will need the following preliminary result.  
\begin{lemma}\label{uniform_bound_un} Suppose that $\epsilon_n = O\left(\frac{1}{  N^t_{n}       \log   (N^s_n)  } \right)$. Then,  there exists $C>0$, independent of $n$,  such that
\begin{align}
\sup_{x \in \mystates_n, \; t \in \time_n}|U_n(x,t)| \leq C, \label{uniform_estimate_un}\\[4pt]
\E_{ \xi_n} \left( \int_{0}^{T}  |\dot{\gamma}(t)|^{q} \d t \right) \leq C.\label{bounded_derivatives_in_expected_value}  
\end{align}
\end{lemma}
\begin{proof} Let us first prove \eqref{uniform_estimate_un}.  Since the cardinality of $\mystates_n$ is equal to $(2(N^s_n)^2 +1)^d$, we have that 

$$\left(\frac{1}{(2(N^s_n)^2 +1)^d}, \hdots, \frac{1}{(2(N^s_n)^2 +1)^d}\right)= \argmin \left\{ \sum_{x \in \mystates_n} p_{x} \log p_{x} \; ; \; p \in \mes(\mystates_n) \right\}.$$
Hence, our assumption over $\epsilon_n$ implies the existence of $\hat{C} >0$, independent of $n$,  such that  for all $x \in \R^d$, $t=t_k$ ($k=0, \hdots, N_n^t-1$), we have
\begin{equation}\label{bound_on_the_entropy_term}\left|\epsilon_n \E_{\xi_{P}^{x,t,n}} \left(  \sum_{k'=k}^{N_{n}^{t}-1} \sum_{y  \in \mystates_n } P(\gamma(t_{k'}),y,t_{k'})\log P(\gamma(t_{k'}),y,t_{k'})\right)\right|\leq \hat{C} \hspace{0.3cm} \forall \; P \in \mathcal{K}_n.
\end{equation} 
Thus, the lower bound is a direct consequence of   the lower bounds for $\ell$ in \eqref{polynomial_growth_ell} and for $f$ and $g$ in \eqref{bounds_on_f_and_g}. In order to obtain the upper bound, choose $P\in \mathcal{K}_n$ in the right hand side of \eqref{NEminimizer1} such that $P(x,x,t_{k'})=1$ for all $k'=k, \hdots, N^{t}_{n}-1$. The bounds in \eqref{polynomial_growth_ell}-\eqref{bounds_on_f_and_g} imply that
$$
U_{n}(x,t_{k}) \leq  (C+C_\ell)\left( T+ 1\right),
$$
and so \eqref{uniform_estimate_un} follows.  Finally, by the lower bound in \eqref{polynomial_growth_ell}, the definition of $\xi_n$,  expression \eqref{NEminimizer1}, estimate  \eqref{uniform_estimate_un}, with $t=0$,  and  \eqref{bounds_on_f_and_g} we have the existence of $C>0$, independent of $n$, such that  
\begin{equation}\label{bound_with_q}
\begin{array}{rcl}
\E_{ \xi_n} \left( \int_{0}^{T}  |\dot{\gamma}(t)|^{q} \d t \right)&=& \E_{\xi_n} \left(\dt_{n}  \sum_{k=0}^{N_{n}^{t}-1}  \left|\frac{ \y(t_{k+1}) - \y(t_k) }{\dt_{n}}\right|^q  \right)\\[8pt]
\; &\leq& \E_{\xi_n} \left(  \frac{\dt_{n}}{\underline{\ell}} \sum_{k=0}^{N_{n}^{t}-1}  \ell\left( \frac{ \y(t_{k+1}) - \y(t_k) }{\dt_{n}} \right) + \frac{C_{\ell}T}{\underline{\ell}} \right)\leq C.
\end{array}
\end{equation}
\end{proof} 
In the proof of the next result, and in the remainder of this article, we set $q':=q/(q-1)$.
\begin{lemma}\label{compactness_sets_gamma_c} Let $C>0$. Then the set 
$$
\Gamma_{C}:= \left\{ \gamma \in W^{1,q}([0,T]; \R^d) \; | \; |\gamma(0)| \leq C \; \; \mbox{{\rm and} } \;  \int_{0}^{T} |\dot{\gamma}(t)|^q \d t \leq C \right\},
$$
is a compact subset of $\Gamma$. 
\end{lemma}
\begin{proof} Let $(\gamma_{n})$ be a sequence in $\Gamma_{C}$. Then, for all $0\leq s\leq t \leq T$,   H\"older's inequality yields
\begin{equation}\label{equicontinuity_gamma_n}
|\gamma_n(t) - \gamma_n(s)|\leq  \int_{s}^{t}|\dot{\gamma}_n(t')| \d t' \leq C^{1/\powL} (t-s)^{1/\powH}.
\end{equation}
Thus, 
\begin{equation}\label{uniform_bound_values_gamma_n} |\gamma_n(t)| \leq |\gamma_n(0)| + |\gamma_n(t)-\gamma_n(0)| \leq C+ C^{1/\powL} T^{1/\powH}.
\end{equation}
As a consequence of \eqref{equicontinuity_gamma_n}-\eqref{uniform_bound_values_gamma_n} and   the Arzel\`a-Ascoli theorem we have existence of $\gamma \in \Gamma$ such that, up to some subsequence,  $\gamma_n \to \gamma$ uniformly in $[0,T]$. Moreover, since $\dot{\gamma}_n$ is bounded in $L^{q}((0,T);\R^d)$ and the function $L^{q}((0,T);\R^d)\ni z \mapsto   \int_{0}^{T} |z(t)|^q \d t\in \R$ is convex and continuous, and hence, weakly lower semicontinuous, we have the existence of $\bar{z} \in L^{q}((0,T);\R^d)$ such that, up to some subsequence, $\dot{\gamma}_n \to \bar{z}$ weakly in $L^{q}((0,T);\R^d)$ and  $ \int_{0}^{T} |\bar{z}(t)|^q \d t \leq \liminf_{n\to \infty} \int_{0}^{T} |\dot{\gamma}_n(t)|^q \d t \leq C$. By passing to the limit in the equality
$$
\gamma_n(t)= \gamma_n(0) + \int_{0}^t \dot{\gamma}_n(s) \d s \hspace{0.4cm} \forall \; t \in [0,T], 
$$
we get that 
$$
\gamma(t)= \gamma(0) + \int_{0}^t \bar{z}(s) \d s \hspace{0.4cm} \forall \; t \in [0,T], 
$$
and, hence, $\gamma \in  W^{1,q}([0,T]; \R^d)$, with $\dot{\gamma}= \bar{z}$ a.e. in $[0,T]$, $|\gamma(0)|\leq C$ and $\int_{0}^{T} |\dot{\gamma}(t)|^q \d t \leq C$. Therefore,  $\gamma \in \Gamma_C$ and, hence, the set $\Gamma_{C}$ is compact. 
\end{proof}
As a consequence of the previous results we easily obtain a compactness property for the sequence $(\xi_n)$. 
\begin{proposition}\label{compactness_of_Pn} Suppose that $\epsilon_n = O\left(\frac{1}{  N^t_{n}       \log   (N^s_n)  } \right)$.  Then, the sequence $(\xi_n)$ is a relatively compact subset of $\mathcal{P}(\Gamma)$ endowed with the topology of narrow convergence. 
\end{proposition}
\begin{proof}  By Prokhorov's theorem  it suffices to show that $(\xi_n)$ is tight, i.e.  we need to prove that  for every $\varepsilon>0$ there exists a compact set $K_{\varepsilon} \subseteq \Gamma$ such that $\sup_{n \in \N} \xi_{n}(\Gamma \setminus K_{\varepsilon}) \leq \varepsilon$. Given  $\varepsilon>0$, the bound \eqref{bound_with_q} and the Markov's inequality yield 
\begin{equation}\label{set_for_compactness}\xi_n \left( \left\{ \gamma \in \Gamma \; \big| \; \gamma \in W^{1,q}((0,T);\R^d) \; \; \mbox{and } \; \; \int_{0}^{T}  |\dot{\gamma}(t)|^{\powL} \d t > \frac{C}{\varepsilon}\right\} \right) \leq \varepsilon \hspace{0.3cm} \forall \; n\in \N.\end{equation}
On the other hand, by {\bf(H3)}{(iii)}, there exists $c_0 >0$ such that for $\xi_n$-almost every $\gamma \in \Gamma$ we have $|\gamma(0)| \leq c_0$. By Lemma \ref{compactness_sets_gamma_c} and \eqref{set_for_compactness},  the set $K_\varepsilon:= \Gamma_{C_\varepsilon}$ with $C_\varepsilon:= \max\{c_0, C/\varepsilon\}$, satisfies the required properties. 
%
%
\end{proof}
Now, we study the compactness of the collection of marginal laws, with respect to the time variables, in the space $C([0,T];\mes_1(\R^d))$. 
\begin{proposition}\label{equicontinuity_Mn_proposition} Suppose that $\epsilon_n = O\left(\frac{1}{  N^t_{n}       \log   (N^s_n)  } \right)$. Then, there exists $C>0$ such that  
\begin{align}
\int_{\R^d} |x|^{\powL} \d M_n(t)(x)=\mathbb{E}_{\xi_n}\left(|\gamma(t)|^\powL \right) \leq C \hspace{0.3cm} \forall \; t \in [0,T], \label{second_order_bounded_moments}\\[3pt]
 d_1  (M_n(t) , M_n (s)) \leq C \abs t-s \abs^{1/\powH} \hspace{0.5cm}\forall \; t,s \in  [0,T],\label{equicontinuity_Mn}
\end{align}
for all $n\in \N$. As a consequence, $M_n \in C([0,T];\mes_1(\R^d))$ for all $n\in \N$ and the sequence  $(M_n)$ is a relatively compact subset of   $\mathcal{C}([0,T] , \mes_1 (\R^d))$.
\end{proposition}

\begin{proof} By definition, for all $t\in [0,T]$ we have that 
\begin{equation}\label{second_order_bounded_pn}
\mathbb{E}_{\xi_n}\left(|\gamma(t)|^\powL \right)\leq 2^{\powL - 1 } \mathbb{E}_{\xi_n}\left(|\gamma(0)|^\powL + T^{\powL/\powH} \int_{0}^{T} |\dot{\gamma}(t)|^\powL \; \d t \right) \leq C,
\end{equation}
for some constant $C>0$, independent of $n$. In the second inequality above we have used that $m_0$ has compact support and \eqref{bound_with_q}. This proves \eqref{second_order_bounded_moments}. 
In order to prove \eqref{equicontinuity_Mn},
by definition of  $d_1$, we have that $d_1(M_{n}(t), M_{n}(s)) \leq d_q(M_{n}(t), M_{n}(s))$ and, setting $\rho_n := (e_{t} , e_{s}) \sharp \xi_n \in \mes(\R^d \times \R^d)$,
$$
\d_\powL^{\powL}  (M_n(t) , M_n (s)) \leq  \int_{\R^d \times \R^d} |x-y|^q \d \rho_n(x,y) =\int_{\Gamma}
\abs \y(t) - \y(s) \abs^{\powL} \; \d \xi_n (\y) 
$$
$$
\leq  |t- s|^{\powL/\powH} \int_{\Gamma} \int_{0}^{T} |\dot{\gamma}(t)|^\powL \d t \; \d \xi_{n}(\y) =  |t - s|^{\powL/\powH} \E_{\xi_n} \left( \int_{0}^{T} |\dot{\gamma}(t)|^\powL \d t \right) \leq C |t - s|^{\powL/\powH},$$
from which \eqref{equicontinuity_Mn} follows.

Finally, relation \eqref{second_order_bounded_moments}  implies that for all $t\in [0,T]$ the set  $\{ M_{n}(t) \; ; \; n\in \N\}$ is relatively compact in $\mathcal{P}_1(\R^d)$ (see \cite[Proposition 7.1.5]{AGS}) and \eqref{equicontinuity_Mn} implies that the family $(M_n)$ is equicontinuous in $C([0,T]; \mes_1(\R^d))$. Therefore,  the last assertion in the statement of the proposition follows from the Arzel\`a-Ascoli theorem.
\end{proof}

Suppose that $\epsilon_n = O\left(1/\left(N^t_{n}\log(N^s_n)\right) \right)$ and let $\xi^\ast \in \mes(\Gamma)$ be a limit point of $(\xi_n)$ (by Proposition \ref{compactness_of_Pn} there exists at least one) and, for notational convenience, we still label by $n \in \N$ a  subsequence of $(\xi_n)$ narrowly converging to $\xi^\ast$. By Proposition \ref{equicontinuity_Mn_proposition}, we have that $(M_n)$ converges to $m(\cdot):= e_{(\cdot)} \sharp \xi^\ast$ in $C([0,T]; \mes_1(\R^d))$.  We now examine the limit behaviour of the corresponding optimal discrete costs $(U_n)$. Defining the Hamiltonian $H: \R^d \to \R$ by 
\begin{equation}\label{definition_hamiltonian}
H(z):= \sup_{z' \in \R^d}\{-z \cdot z' - \ell(z')\} \hspace{0.4cm} \forall \; z\in \R^d,\end{equation}
and assuming that $\epsilon_n = o\left(1/\left(N^t_{n}\log(N^s_n)\right) \right)$,  in Proposition \ref{convergence_of_the_value_functions} we prove that $(U_n)$ converges, in a suitable sense, to a viscosity solution of 
\begin{equation}\label{limit_hjb}
\begin{array}{rcll}
-\partial_t u + H(\nabla u) &=&f(x,m(t))   \qquad &x\in \R^d, \; \; t \in (0,T), \\[6pt]
u(x,T)&=& g(x,m(T)) \qquad &x\in \R^d.
\end{array}
\end{equation}
Classical results imply that under {\bf(H3)}{\rm(i)}-{\rm(ii)}  equation \eqref{limit_hjb} admits at most one viscosity solution (see e.g. \cite[Theorem 2.1]{MR2784834}). In \cite[Proposition 1.3 and Remark 1.1]{MR1485734} the existence of a viscosity solution $u$ is proved,  as well the following representation formula: for all $(x,t) \in \R^d \times (0,T)$
\begin{equation}\label{representation_formula_u}
u(x,t)=\inf \left\{ \int_{t}^{T} \left[ \ell(\dot{\y}(s)) + f(\y(s) ,m(s)) \right] \d s + g(\y(T) ,m(T)) \; \; \big| \; \; \y \in W^{1,q}([0,T]; \R^d), \; \; \y(t)= x \right\}.
\end{equation}
Standard arguments using   \eqref{representation_formula_u} show that $u$ is continuous in $\R^d\times [0,T]$ (see e.g. \cite[Theorem 2.1]{MR1485734}).
\begin{remark}\label{redefinition_of_mfg_equilibrium} Definition \ref{definition_mfg_equilibrium} can thus be rephrased as follows:   $\xi^* \in \mes_{m_0}(\Gamma)$ is a {\it MFG equilibrium for \eqref{family_of_optimal_control_problems}}   if $[0,T] \ni t \mapsto m(t):= e_t \sharp \xi^*$ belongs to $C([0,T]; \mes_1(\R^d))$ and for $\xi^{\ast}$-almost all $\gamma$ we have that 
\begin{equation}\label{value_function_minimum_attained}u(\gamma(0),0)= \int_{0}^{T}  \left[\ell(\dot{\y}(t))   +  f(\y(t),m (t)) \right] \; \d t + g(\y(T) , m(T)),
\end{equation}
where $u$ is the unique viscosity solution to \eqref{limit_hjb}.
\end{remark}\medskip

In order to prove the convergence of $U_n$ to $u$, we will need the following auxiliary functions  
\begin{equation}\label{definition_of_the_envelopes}
U^{\ast}(x,t):= \limsup_{\substack{n\to \infty \\ \mystates_n \ni y \to x \\ \time_n \ni s \to t }} U_n(y,s), \hspace{1cm} U_{\ast}(x,t):= \liminf_{\substack{n\to \infty \\ \mystates_n \ni y \to x \\ \time_n \ni s \to t }} U_n(y,s) \hspace{0.8cm} \forall \; x \in \R^d, \; \; t\in [0,T].
\end{equation}
By Lemma \ref{uniform_bound_un}, the functions $U^{\ast}$ and $U_{\ast}$ are well defined if  $\epsilon_n = O\left( 1/(N^t_{n}\log(N^s_n)) \right)$. 
In some of the next results, we will need to assume  a stronger hypothesis on $\epsilon_n$, namely $\epsilon_n = o\left( 1/(N^t_{n}\log(N^s_n)) \right)$, which will allow us to eliminate the entropy term in the limit.

Before proving the convergence of the value functions, we will need a preliminary result. 
\begin{lemma}\label{envelope_properties}  Assume  that $\epsilon_n = O\left(\frac{1}{N^t_{n}       \log   (N^s_n)} \right)$. Then,\\
{\rm(i)} $U^\ast$ and $U_\ast$ are upper and lower semicontinuous, respectively.\\
{\rm(ii)} If in addition, $\epsilon_n = o\left(\frac{1}{N^t_{n}       \log   (N^s_n)} \right)$, we have that $U^\ast(x,T)=U_\ast(x,T)=g(x,m(T))$ for all $x\in \R^d$. 
\end{lemma}
\begin{proof} The proof of assertion {\rm(i)} is the same than the proof of \cite[Chapter V, Lemma 1.5]{MR1484411}. Let us prove {\rm(ii)}. For $n \in \N$,  let $x^n \in \mystates_n$, $t^n \in \time_n$ and $k: \N \to \N$ such that $t^{n}=t_{k(n)}$ (recall that $\time_n= \{0, t_1,\hdots, t_{N^t_n}\}$).  
Because of our assumption on $\epsilon_n$, we can write \small
\begin{equation}\label{Un_writen_in_terms_of_sums}
\begin{array}{ll}
U_{n}(x^n,t^n) 
=& \sum_{\y \in \Gamma_{x^n,t^n}^{\mystates_n,\time_n }}  p_{\hat{P}_{n}}^{x^n,t^n}(\gamma)  \left( \sum_{k=k(n)}^{N_n^t-1} \dt_{n}  \left[   \ell\left( \frac{  \y(t_{k+1}) -  \y(t_{k}) }{\dt_{n}}\right)+  f(\y(t_k)  , M_n (t_k) ) \right] + g (\y(T),M_n (T)) \right)\\[10pt]
\; &  + o(1),
\end{array}
\end{equation} \normalsize
where we recall that $p_{\hat{P}_{n}}^{x^n,t^n}$ is defined in \eqref{dirac_masses_in_gamma}.
%
Using the definition of $U_n$ and arguing as in the proof of Lemma \ref{uniform_bound_un}, we have that
$$
\sum_{k=k(n)}^{N_n^t-1}  \dt_{n} f(\y(t_k)  , M_n (t_k) ) =  O (T-t^n),$$
$$
U_{n}(\pos^n,t^n) \leq  g(x^n , M_n (T))+O(T-t^n) +o(1).$$
Therefore, if $x^n \to x \in \R^d$ and $t^n\to T$, we have 
$$\limsup_{n\to \infty } \; U_n (x^n,t^n) \leq
g(x, m(T)), $$
from which we deduce that $U^{\ast}(x,T) \leq g(x,M(T))$ for all $x\in \R^d$. Next, for every $\y \in \Gamma_{x^n,t^n}^{\mystates_n,\time_n }$ we have
$$
\left| \gamma(T)-x_n\right|^q \leq  \left(\sum_{k=k(n)}^{N_n^t-1} |\gamma(t_{k+1})-\gamma(t_{k+1})|\right)^q \leq (N_n^t-k(n))^{q-1}\sum_{k=k(n)}^{N_n^t-1}| \gamma(t_{k+1})-\gamma(t_{k+1})|^q,
$$
which implies that 
\begin{equation}\label{bounds_from_below_un1}
 \sum_{k=k(n)}^{N_n^t-1}    \dt_{n} \left| \frac{  \y(t_{k+1}) -  \y(t_{k}) }{\dt_{n}}\right|^q  \geq  \frac{\dt_n}{  (N_n^t - k(n))^{\powL - 1}}\left|\frac{  \y(T) -  x_n }{\dt_n}\right|^q = \frac{1}{ (T- t^n)^{\powL - 1}}\left|   \y(T) -  x_n \right|^q.
\end{equation}
%
Thus, setting   $p_{T,\posn}^{\pos^n,t^n} :=  \xi_{\hat{P}_n}^{x^n,t^n}   ( \{ \gamma \in \Gamma_{t^n} \; | \; \y(T) = \posn\})$, the bounds \eqref{polynomial_growth_ell}, \eqref{bounds_on_f_and_g}, \eqref{bounds_from_below_un1} and equation \eqref{Un_writen_in_terms_of_sums} yield 
\begin{equation}\label{inequalities_to_obtain_the_liminf}
\begin{split}
U_n (x^n,t^n)
&\geq \sum_{\posn \in \mystates_n } p_{T,\posn}^{\pos^n,t^n} \left( \frac{ \underline{\ell}\left| \posn -  \pos^n \right|^\powL }{  (T - t^n)^{\powL-1}}  + g ( \posn ,M_n (T)) \right) + O(T-t^n) + o(1)  \\
&\geq \min_{\posn \in \mystates_n } \left\{\frac{ \underline{\ell}\left| \posn -  \pos^n \right|^\powL }{ (T - t^n)^{\powL-1}}  + g ( \posn ,M_n (T)) \right\} + O(T-t^n)+ o(1) .
\end{split}
\end{equation}
Suppose that $\posn^*_n$ minimizes the ``$\min$'' term in the  last line above. By definition, we have
$$ \frac{ \underline{\ell}\left| \posn_n^* -  \pos^n \right|^\powL }{   (T - t^n)^{\powL-1}}  \leq  g(x^n,M_n(T)) - g ( \posn^*_n ,M_n (T)) \leq C \left| \posn^*_n -  \pos^n \right|,$$
where the last inequality follows from \eqref{bounds_on_f_and_g}.  As a consequence, we get that $\left| \posn^*_n -  \pos^n \right| = O(T-t^n)$ and so $\frac{ \left| \posn_n^* -  \pos^n \right|^\powL }{  (T - t^n)^{\powL-1}}  \to 0$ as $n\to \infty$. Therefore, as $n\to \infty$,
$$
\min_{\posn \in \mystates_n} \left\{ \frac{ \underline{\ell} \left| \posn -  \pos^n \right|^\powL }{  (T - t^n)^{\powL-1}}  + g ( \posn ,M_n (T)) \right\}
=
\frac{ \underline{\ell}\left| \posn_n^* -  \pos^n \right|^\powL }{   (T - t^n)^{\powL-1}} + g ( \posn^*_n ,M_n (T))
\to 
g ( \pos , m  (T)).
$$
By \eqref{inequalities_to_obtain_the_liminf}, this implies that 
$$\liminf_{n \to \infty } U_n (x^n,t^n) \geq
g(x, m(T)),$$
from which we deduce that $U_{\ast}(x,T) \geq g(x,m(T))$. The result follows.
\end{proof}
Now, we prove the convergence of the sequence $(U_n)$. The argument of the proof uses some ideas from the theory of approximation of viscosity solutions (see e.g. \cite{MR1115933}).
\begin{proposition}\label{convergence_of_the_value_functions} Assume  that, as $n\to \infty$,  $N^t_n/N^s_n \to 0$  and $\epsilon_n = o\left(\frac{1}{N^t_{n} \log   (N^s_n)} \right)$. Then,  $U^\ast=U_{\ast}=u$, where $u$ is given by \eqref{representation_formula_u}, or equivalently, where $u$ is the unique continuous viscosity solution to  \eqref{limit_hjb}. As a consequence, for every compact set $Q \subseteq \R^d$ we have that 
\begin{equation}\label{uniform_convergence_on_grid_points}
\sup_{(x,t) \in \left(\mystates_n \cap Q\right) \times \time_n} |U_n(x,t)-u(x,t)| \to 0 \hspace{0.5cm} \mbox{as $ \; n\to \infty$}.
\end{equation}
\end{proposition}
\begin{proof}
Let us prove that $U^*$ is a viscosity subsolution of equation \eqref{limit_hjb}. Let $\phi \in \C^1( \R^d \times [0,T])$  and $(x^*,t^*)\in \R^d \times (0,T)$ be such that $(x^*,t^*)$ is a local maximum of $U^\ast - \phi$ on $\R^d \times (0,T)$.

By standard arguments in the theory of viscosity solutions (see e.g. \cite[Chapter II]{MR1484411}), we may assume that $\phi$ is bounded as well as its time and space derivatives  and that $(x^*,t^*)$ is a strict global maximum of $U^\ast - \phi$. Arguing as in the proof of \cite[Chapter V, Lemma 1.6]{MR1484411}, we can show the existence of a sequence $(x^n, t^n)$ in $ \mystates_n\times \time_n$ such that  $ (x^n, t^n)\to (x^\ast, t^\ast)$, $U_n(x^n,t^n) \to U^\ast(x^\ast, t^\ast)$ and $U_n-\phi$ has maximum at $(x^n,t^n)$ in the set $\left(\mystates_n\times \time_n\right) \cap B_\delta$, where $B_\delta:=\{ (x, t) \in \R^d \times (0,T) \; ; \; |x-x^{\ast}|+|t-t^\ast| \leq \delta \}$ and $\delta>0$ is such that $B_\delta \subseteq \R^d \times (0,T)$.
%


Now, let $\xi \in C^{\infty}(\R^d \times [0,T])$ be such that $0\leq \xi \leq 1$,  $\xi(x,t)= 0$ if $(x,t)\in B_{\frac{\delta}{2}}$ and $\xi(x,t)= 1$ if $(x,t)\in \left(\R^d \times (0,T) \right) \setminus B_{\delta}$. Then, using that $U_n$ and $\phi$ are bounded, we can choose $M>0$ large enough such that, setting $\bar{\phi}:= \phi +M \xi$, the function $U_n- \bar{\phi}$ has   maximum in $\mystates_n\times \time_n$ at the point $(x^n,t^n)$. Note that $\partial_{t}\bar{\phi}(x^\ast,t^\ast)=\partial_{t}\phi(x^\ast,t^\ast)$ and  $\nabla \bar{\phi}(x^\ast, t^\ast)= \nabla \phi(x^\ast, t^\ast)$.  

As in the proof of Lemma \ref{envelope_properties}, let $k: \N \to \N$ be such that $t^n=t_{k(n)}$.  Since $U_n(x^n,t^n)$ satisfies
%
\begin{equation*}
\begin{split}
U_{n}(\pos^n,t^n)
&= \min_{p \in \mes(\mystates_n)} \sum_{\posn\in \mystates_n} p(y) \left( \dt_{n} \ell \left(\frac{ \posn - \pos^n}{\dt_{n}}\right) + \dt_{n} f(x^n ,M_n (t^n) ) + U_{n}(\posn,t_{k(n)+1}) \right) + \epsilon_n \mathcal{E}_n(p),
\end{split}
\end{equation*}
and $U_{n}(\posn,t_{k(n)+1})-U_n(x^n,t^n) \leq \bar{\phi}(y,t_{k(n)+1})- \bar{\phi}(x^n,t^n)$ for all $y\in \mystates_n$, we have that 
\begin{equation}\label{ViscosityInequalityFiniteMFGEliminatingEntropy}
\begin{split}
0 &\leq \min_{p \in \mes(\mystates_n) } \sum_{\posn\in \mystates_n} p(y) \left( \dt_{n} \ell \left(\frac{ \posn - \pos^n}{\dt_{n}}\right) + \dt_{n} f(x^n ,M_n (t_{k(n)}) ) + \bar{\phi}(\posn,t_{k(n)+1}) - \bar{\phi}(\pos^n,t_{k(n)}) \right) + \epsilon_n \mathcal{E}_n(p) ,\\
&\leq \min_{ \posn \in \mystates_n} \left\{\dt_{n} \ell \left(\frac{ \posn - \pos^n}{\dt_{n}}\right) + \dt_{n} f(x^n , M_n (t_{k(n)}) ) + \bar{\phi}(\posn,t_{k(n)+1}) - \bar{\phi}(\pos^n,t_{k(n)}) \right\}+ \epsilon_n \mathcal{E}_n(p),
\end{split}
\end{equation}
where the second inequality follows from the first one by taking   for each $y\in \mystates_n$ the vector $p\in \mes (\mystates_n)$ defined as $p(z)=1$ iff $z=y$.     Dividing by $\Delta t_n$ and recalling that $\epsilon_n=o\left(\frac{1}{N^t_{n} \log   (N^s_n)} \right)$, we get
$$0 \leq f(x^n ,M_n (t^n) ) +  \min_{\posn \in \mystates_n} \left\{  \ell \left(\frac{ \posn - \pos^n}{\dt_{n}}\right)  + \frac{\bar{\phi}(\posn,t_{k(n)+1}) - \bar{\phi}(\pos^n,t_{k(n)})}{\dt_{n}} \right\}+o(1),
$$
and so, taking liminf, 
\begin{equation}\label{1conver}
0 \leq f(x^\ast,m (t^\ast))+   \liminf_{n\to \infty}   \min_{\posn \in \mystates_n} \left\{  \ell \left(\frac{ \posn - \pos^n}{\dt_{n}}\right) + \frac{\bar{\phi} (\posn,t_{k(n)+1}) - \bar{\phi} (x^n,t_{k(n)})}{\dt_{n}} \right\} ,
\end{equation}
where we have used that $M_n \to m$ in $C([0,T]; \mes_1(\R^d))$.
Let us study the second term in the right hand side above.  For fixed $n$, let $y_n^\ast$ be such that
$$
y_n^\ast \in  \argmin_{\posn \in \mystates_n} \left\{  \ell \left(\frac{ \posn - \pos^n}{\dt_{n}}\right) + \frac{\bar{\phi} (\posn,t_{k(n)+1}) - \bar{\phi} (x^n,t_{k(n)})}{\dt_{n}} \right\},
$$
or equivalently, setting $\alpha_n^\ast:= \frac{y_n^\ast-x^n}{\dt_n} $,
\begin{equation}\label{minimization_z_n}
\ell\left( \alpha_n^\ast\right)+ \frac{\bar{\phi} (x^n +\dt_n \alpha_n^\ast,t_{k(n)+1}) - \bar{\phi} (x^n,t_{k(n)})}{\dt_{n}} \leq    \ell \left(\frac{ \posn - \pos^n}{\dt_{n}}\right) + \frac{\bar{\phi} (\posn,t_{k(n)+1}) - \bar{\phi} (x^n,t_{k(n)})}{\dt_{n}},
\end{equation}
for all $y\in \mystates_n$. By taking $y=x^n$ in the expression above, using that $\partial_t \bar{\phi}$ and $\nabla \bar{\phi}$ are bounded and the growth condition \eqref{polynomial_growth_ell} on $\ell$, we obtain that the sequence $(\alpha_n^\ast)$ is bounded.  Let $\alpha^\ast$ be a cluster point of this sequence and consider a subsequence of $(\alpha_n)$, still indexed by $n$, such that $\alpha_n^\ast\to \alpha^\ast$.  The condition   $N^t_n/N^s_n \to 0$  implies that for any $\alpha \in \R^d$ we can find a sequence $(y^n)$ in  $\mystates_n$ such that $\frac{y^n-x^n}{\dt_n} \to \alpha$ as $n\to \infty$. Taking $y=y^n$ in \eqref{minimization_z_n} and passing to the limit  yields 
\begin{equation}\label{minimization_z_limit}
\ell( \alpha^\ast) + \nabla \phi(x^\ast,t^\ast)\cdot \alpha^\ast  \leq \ell(\alpha)   + \nabla \phi(x^\ast,t^\ast)\cdot \alpha \hspace{0.3cm} \forall \; \alpha \in \R^d, 
\end{equation}
which  implies, by the definition of $H$ in \eqref{definition_hamiltonian},  that 
$$
-\ell( \alpha^\ast)- \nabla \phi(x^\ast,t^\ast)\cdot \alpha^\ast=  H(\nabla \phi(x^\ast,t^\ast)).
$$
Since the previous equality holds for any cluster point of $\alpha_n$, we deduce that 
$$\liminf_{n\to \infty}   \min_{\posn \in \mystates_n} \left\{  \ell \left(\frac{ \posn - \pos^n}{\dt_{n}}\right) + \frac{\bar{\phi} (\posn,t_{k(n)+1}) - \bar{\phi} (x^n,t_{k(n)})}{\dt_{n}} \right\}=- H(\nabla \phi(x^\ast,t^\ast))+ \partial_t \phi(x^\ast,t^\ast),$$
and, hence, \eqref{1conver} gives
%
$$
-\partial_t \phi(x^\ast,t^\ast) +  H(\nabla \phi(x^\ast,t^\ast)) \leq f(x^\ast,m(t^\ast)),
$$
which proves that $U^\ast$ is a subsolution to \eqref{limit_hjb}. An analogous argument shows that $U_\ast$ is a supersolution to \eqref{limit_hjb}.
Assumptions  {\bf(H3)}{\rm(i)-\rm(ii)} ensure a comparison principle for  \eqref{limit_hjb} (see \cite[Theorem 2.1]{MR2784834}). Therefore, since $U^\ast(\cdot,T)= U_\ast(\cdot,T)$ by Lemma \ref{envelope_properties}{\rm(ii)}, we have  that $U^\ast=U_\ast=u$ as announced. Using this result, the proof of \eqref{uniform_convergence_on_grid_points} is identical to the proof of \cite[Chapter V, Lemma 1.9]{MR1484411}.
\end{proof}
We have now all the elements to prove the main result in this article. We will need an additional assumption over $\ell$, $f$ and $g$.\medskip\\
{\bf(H4)} We assume that: \smallskip\\
{\rm(i)} The function $\ell$ is convex.  \smallskip\\
{\rm(ii)} There exists  $C>0$ and a modulus of continuity $\omega:[0,+\infty) \to [0,+\infty)$  such that for $h=f$, $g$ we have
\begin{equation}\label{additional_assumption} 
|h(x,m) - h(x,m')| \leq C(1+ |x|^q) \omega\left(d_{1}(m,m')\right) \hspace{0.3cm} \forall \; x\in \R^d, \; m, \; m' \in \mes_1(\R^d). 
\end{equation}  \vspace{0.0001cm}

\begin{theorem}\label{ConvergenceMainTheorem} Suppose that {\bf(H3)}-{\bf(H4)} hold and that, as $n\to \infty$,  $N^t_n/N^s_n \to 0$ and $\epsilon_n = o\left(\frac{1}{N^t_{n}\log   (N^s_n)} \right)$. Then, the following assertions hold true:  \smallskip\\
{\rm(i)} There exists at least one limit point $\xi^\ast$ of $(\xi_n)$, with respect to the narrow topology in $\mes(\Gamma)$, and every such limit point   is a MFG equilibrium for \eqref{family_of_optimal_control_problems}.\smallskip\\
{\rm(ii)} Consider any converging subsequence of $(\xi_{n'})$ of $(\xi_{n})$, with limit $\xi^{\ast} \in \mes(\Gamma)$, and let $(U_{n'}, M_{n'})$ be the associated solutions to \eqref{introsFiniteMFGNashEinConvQuestion}. Denote by $u$ be the unique viscosity solution to \eqref{limit_hjb} with $m(t):=e_t \sharp \xi^\ast$ for all $t\in [0,T]$. Then, the sequence $(M_{n'})\subseteq C([0,T]; \mes_1(\R^d))$, defined by \eqref{interpolations}, converge to $m$ in $C([0,T]; \mes_1(\R^d))$ and \eqref{uniform_convergence_on_grid_points} holds for $(U_{n'})$ and $u$.
 
%
%
%
%
\end{theorem}

\begin{proof} Assertion {\rm(ii)} is a straightforward consequence of the first assertion and  Proposition \ref{convergence_of_the_value_functions}, hence, we only need to prove {\rm(i)}.  The existence of at least one limit point $\xi^\ast$ of $(\xi_n)$ is a consequence of Proposition \ref{compactness_of_Pn}. Let us still index by $n$ a subsequence of $(\xi_n)$ narrowly converging to $\xi^\ast$. By Proposition \ref{equicontinuity_Mn_proposition}, we have that $m(\cdot):= e_{(\cdot)} \sharp \xi^\ast$ is the limit in $C([0,T]; \mes_1(\R^d))$ of $M_n$. 
By definition of $\xi_n$ and our condition over $\epsilon_n$, we have
\begin{equation}\label{eq::18}
\E_{\xi_n } \left(  \int_{0}^{T}  \left[\ell(\dot{\y}(t))   +  f(\y([t]_{\time_n}),M_n ([t]_{\time_n})) \right] \; \d t + g(\y(T) , M_n (T)) \right) + o(1) = \sum_{x \in \mystates_n} U_n(x,0) M_{n,0}(x)  ,
\end{equation} 
where $[t]_{\time_n}$ is the greatest element in $\time_n$ not larger than $t$. Using that the support of $M_{n,0}$ is uniformly bounded and relation \eqref{uniform_convergence_on_grid_points} in Proposition \ref{convergence_of_the_value_functions}, we easily get that the right hand side above converges to $\int_{\R^d} u(x,0) \d m_0(x)= \E_{\xi^\ast}\left( u(\gamma(0),0)\right)$, where $u$ is the unique viscosity solution to \eqref{limit_hjb}.  On the other hand, arguing as in the proof of Lemma \ref{compactness_sets_gamma_c}, the lower bound in \eqref{polynomial_growth_ell} and the convexity of $\ell$ imply that the mapping 
$$   \Gamma \ni \gamma \mapsto  
\begin{cases}
\int_{0}^{T} \ell(\dot{\y}) \; \d t , \quad &\text{if } \y \in W^{1,q}([0,T] ; \R^d), \\
+\infty  &\text{otherwise,}
\end{cases}
$$
is  lower semicontinuous. Therefore, by \cite[Lemma 5.1.7]{AGS} and \eqref{bounded_derivatives_in_expected_value}, we have
\begin{equation}\label{eq::19}
 \E_{\xi^\ast} \left(  \int_{0}^{T}\ell(\dot{\y}(t)) \; \d t \right) \leq \liminf_{n\to \infty} \E_{\xi_n } \left(  \int_{0}^{T} \ell(\dot{\y}(t)) \; \d t\right)<\infty,
\end{equation}
which, together with the lower bound in \eqref{polynomial_growth_ell}, implies that the support of $\xi^\ast$ is contained in $W^{1,q}([0,T];\R^d)$.  By assumption  {\bf(H3)}{\rm(ii)}, for all $k=0, \hdots, N^t_{n}-1$ we have that 
\begin{equation}\label{estimate_difference_of_f_small_interval}
\left|\E_{\xi_n}\left(\int_{t_{k}}^{t_{k+1}}\left[ f(\y(t_k),M_n (t_k))-f(\y(t),M_n (t_k)) \right] \d t \right) \right| \leq C\E_{\xi_n}\left(\int_{t_{k}}^{t_{k+1}} \left|\y(t)-\y(t_k)\right| \d t \right).
\end{equation}
Since $\gamma(t)= \gamma(t_k) + \dot{\gamma}(t)(t-t_k)$ for $\xi_n$-almost all $\gamma$ and all $t\in (t_k,t_{k+1})$, the bound \eqref{bounded_derivatives_in_expected_value} gives
$$\E_{\xi_n}\left(\int_{t_{k}}^{t_{k+1}} \left|\y(t)-\y(t_k)\right| \d t \right)= \Delta t_n (\Delta t_n)^{\frac{1}{q'}} \left[\E_{\xi_n}\left( \int_{0}^{T} |\dot{\gamma}(t)|^q \d t \right) \right]^{\frac{1}{q}}\leq C(\Delta t_n)^{1+\frac{1}{q'}}, $$
for some constant $C>0$. Thus,  by \eqref{estimate_difference_of_f_small_interval}, 
$$
\E_{\xi_n } \left(  \int_{0}^{T} f(\y([t]_{\time_n}),M_n ([t]_{\time_n}))   \d t \right) = \E_{\xi_n } \left(  \int_{0}^{T} f(\y(t),M_n ([t]_{\time_n}))   \d t \right) + o(1).
$$
The relation above and \eqref{additional_assumption} yield 
\begin{equation}\label{limit_term_with_f}
\begin{array}{rl}
\E_{\xi_n } \left(  \int_{0}^{T} f(\y([t]_{\time_n}),M_n ([t]_{\time_n}))   \d t \right) =&  \E_{\xi_n } \left(  \int_{0}^{T} f(\y(t),  m(t))   \d t \right)\\[8pt]
\; &+ C\left( 1+ \sup_{t\in [0,T]} \E_{\xi_n }(|\y(t)|^q) \right)\sup_{t\in [0,T]}\omega\left(\d_1 (M_n ([t]_{\time_n}) , m(t))\right)\\[8pt]
\; &  + o(1)\\[8pt]
= & \;   \E_{\xi_n } \left(  \int_{0}^{T} f(\y(t),m(t))   \d t \right) + o(1), 
\end{array}
\end{equation}
where, in the last equality, we have used \eqref{second_order_bounded_moments} and the fact that $M_n \to m$ in $C([0,T]; \mes_1(\R^d))$. Analogously, 
\begin{equation}\label{limit_term_with_g}
\E_{\xi_n } \left( g(\y(T) , M_n (T)) \right)=\E_{\xi_n } \left( g(\y(T) , m (T)) \right) + o(1).
\end{equation}
Therefore, passing to the limit $n \to \infty$ in \eqref{eq::18} and using \eqref{eq::19}, \eqref{limit_term_with_f} and \eqref{limit_term_with_g}, we get
\begin{equation}\label{limit_inequality}
\E_{\xi^\ast} \left(  \int_{0}^{T}  \left[\ell(\dot{\y}(t))   +  f(\y(t),m (t)) \right] \; \d t + g(\y(T) , m(T)) \right) \leq  \E_{\xi^\ast}\left( u(\gamma(0),0)\right).
\end{equation}
Since, by definition, 
$$u(\gamma(0),0) \leq  \int_{0}^{T}  \left[\ell(\dot{\y}(t))   +  f(\y(t),m (t)) \right] \; \d t + g(\y(T) , m(T)) \hspace{0.3cm} \forall \; \y \in W^{1,q}([0,T];\R^d),$$
inequality \eqref{limit_inequality} implies that for $\xi^{\ast}$-almost all $\gamma$ we have that 
$$u(\gamma(0),0)= \int_{0}^{T}  \left[\ell(\dot{\y}(t))   +  f(\y(t),m (t)) \right] \; \d t + g(\y(T) , m(T)),$$
i.e. $\xi^\ast$ is a MFG equilibrium for \eqref{family_of_optimal_control_problems} (see Remark \ref{redefinition_of_mfg_equilibrium}).
\end{proof} 

Finally, let us recall the relationship between the MFG equilibrium $\xi^\ast$, defined in terms of probability measures on $\Gamma$ in Definition \ref{definition_mfg_equilibrium},  and the first order MFG system introduced by Lasry and Lions in \cite[Section 2.5]{LL07mf}. The latter is given by
$$\left.\begin{array}{rcl} - \partial_t u  + H(\nabla u) &=& f(x, m(t) ) \hspace{0.3cm} \mbox{in } \R^d \times (0,T), \\[6pt]
					\partial_t m - \mbox{div}\left(  \nabla H(\nabla u) m  \right) &=& 0 \hspace{0.3cm} \mbox{in } \R^d \times (0,T), \\[6pt]
					u(\cdot,T)=g(\cdot,m(T))\hspace{0.2cm} \mbox{in $\R^d$}, & \; & m(0)= m_0.
\end{array} \right\} \eqno(MFG)
$$
We say that $(u,m)$ solves $(MFG)$ if $u$ is continuous, Lipschitz w.r.t. its first argument, $m \in \C([0,T]; \mes_1(\R^d))$, the first equation is satisfied in the viscosity sense and the second one  is satisfied in the sense of distributions. 

We will need the additional assumption  \medskip\\
{\bf(H5)} The following assertions hold true: \smallskip\\
{\rm(i)} The function $\ell$ is $\C^2$, the growth condition \eqref{polynomial_growth_ell} is satisfied with $q=2$,  and for  $h=f$, $g$ we have that  $h(\cdot, m)$ is $\C^2$, for every $m\in \mes_1(\R^d)$, and there exists $C>0$ such that
\begin{equation}\label{bounds_on_f_and_g} \sup_{m \in \mes_1 (\R^d) } \left\{   \norm h(\cdot , m)\norm_{\infty} + \norm D_x h(\cdot , m) \norm_{\infty}+\norm D_{xx}^2 h(\cdot , m) \norm_{\infty} \right\} \leq C.
\end{equation} 
{\rm(ii)} The initial distribution $m_0$ is absolutely continuous and its density  belongs to $L^{\infty}(\R^d)$. 
\smallskip\\

Under assumptions {\bf(H3)} and  {\bf(H5)}, there exists at least one solution $(u,m)$ to $(MFG)$ (see \cite{LL07mf,Cdga13}). Moreover, this solution is unique under the following monotonicity assumption on $f$ and $g$:  
\begin{equation}\label{monotonicity_continuous_mfg}
\mbox{For $h=f$, $g$, we have  } \hspace{0.2cm}\int_{\R^d}\left[h(x,m)  -h(x,m') \right] \d (m-m')(x) \geq 0   \hspace{0.5cm} \forall \; m, \; m' \in \mes_1(\R^d).
\end{equation}
If $(u,m)$ is a solution of $(MFG)$, the results in \cite[Chapter 6]{CannarsaSinestrari} imply that for almost all $x\in \R^d$ the equality \eqref{representation_formula_u} holds and 
\begin{equation}\label{gamma_x_minimizer}
u(x,0)= \int_{0}^{T}  \left[\ell(\dot{\overline \y}^x(t))   +  f(\overline \y^x(t),m (t)) \right] \; \d t + g(\overline \y^x(T) , m(T)),
\end{equation}
where $\overline \gamma^x$ is the unique solution to 
\begin{equation}\label{equation_feedback_form}
\dot{\gamma}(t)= - \nabla H(\nabla u(\gamma(t), t)) \; \; \; t \in (0,T), \hspace{0.5cm} \gamma(0)=x. 
\end{equation}
Moreover, $\overline \gamma^x$ is the only curve in $W^{1,2}([0,T]; \R^d)$ such that \eqref{gamma_x_minimizer} holds. By considering a measurable selection of the set 
\begin{equation}\label{set_of_curves_to_be_chosen}\left\{ \gamma^x \in W^{1,2}([0,T]; \R^d) \; | \; \gamma^x \; \mbox{satisfies \eqref{gamma_x_minimizer}}, \; x\in \R^d\right\},
\end{equation}
(and so $\gamma^x= \overline{\y}^x$ for a.e. $x\in \R^d$) and using that the second equation in $(MFG)$ admits a unique solution (thanks to \cite[Theorem 8.2.1]{AGS}) if we define $\xi^{\ast}:= \gamma^{(\cdot)} \sharp m_0 \in \mes(\Gamma)$, we have that $\xi^\ast$ is a MFG equilibrium in the sense of Definition \ref{definition_mfg_equilibrium}. Conversely, given a MFG equilibrium $\xi^\ast$, setting $m(t):= e_t \sharp \xi^\ast$ and defining $u$ by \eqref{representation_formula_u},   the first equation in $(MFG)$ and the boundary condition at time $T$ are satisfied. Moreover,  the results in  \cite[Chapter 6]{CannarsaSinestrari} imply that $\xi^{\ast}:= \gamma^{(\cdot)} \sharp m_0 \in \mes(\Gamma)$, where $\gamma^{(\cdot)}$ is a measurable selection of curves in \eqref{set_of_curves_to_be_chosen}. Therefore,   $m$ solves the second equation in $(MFG)$ in the distributional sense. 

By the previous remarks, we have the following consequence of Theorem \ref{ConvergenceMainTheorem}. 
\begin{corollary}\label{convergence_to_mfg_solution_pde_system} Suppose that {\bf(H3)}, {\bf(H4)} and {\bf(H5)} hold and that, as $n\to \infty$,  $N^t_n/N^s_n \to 0$ and $\epsilon_n = o\left( 1/\left(N^t_{n}\log   (N^s_n)\right) \right)$.  Then,   associated  to every   limit point $m$ of $(M_n)$ in $C([0,T]; \mes_1(\R^d))$  {\rm(}there exists at least one{\rm)}, there exists $u \in C( \R^d \times [0,T] )$, Lipschitz w.r.t. its first variable,  and a subsequence of $(M_n, U_n)$, which we still index by $n$, such that $(u,m)$ solves $(MFG)$,  $M_n \to m$ in $C([0,T]; \mes_1(\R^d))$ and $(U_n,u)$ satisfies \eqref{uniform_convergence_on_grid_points}. 
\end{corollary}

\begin{remark} Under the previous assumptions,  the convergence results in Theorem \ref{ConvergenceMainTheorem} and in Corollary \ref{convergence_to_mfg_solution_pde_system} hold for the entire sequence {\rm(}i.e. without need of extracting a subsequence{\rm)} if the solution to $(MFG)$ is unique. This holds true under the following monotonicity assumption on $h=f$, $g$ {\rm(}see {\rm\cite{LL07mf})}
$$
\int_{\R^d} \left( h(x,m) - h(x,m') \right) \d (m-m')(x) \geq 0 \hspace{0.5cm} \forall \; m, \; m' \in \mes_1(\R^d). 
$$
\end{remark}

\def\bibname{References}
\bibliographystyle{abbrv}
\bibliography{biblio}
\end{document}